\documentclass[11pt,a4paper]{article}

\usepackage[utf8]{inputenc}
\usepackage[T1]{fontenc}
\usepackage{lmodern}
\usepackage{microtype}
\usepackage[margin=2.65cm]{geometry}
\usepackage{amsmath,amssymb,amsthm,mathtools}
\usepackage{enumitem}
\usepackage{booktabs}
\usepackage{tabularx,array}
\usepackage{xcolor}
\usepackage[colorlinks=true,linkcolor=blue!45!black,citecolor=blue!45!black,urlcolor=blue!45!black]{hyperref}
\usepackage[nameinlink,noabbrev]{cleveref}
\usepackage{tikz}
\usepackage{float}
\usetikzlibrary{arrows.meta,calc,decorations.pathreplacing}

\newtheorem{theorem}{Theorem}[section]
\newtheorem{proposition}[theorem]{Proposition}
\newtheorem{corollary}[theorem]{Corollary}
\newtheorem{lemma}[theorem]{Lemma}
\theoremstyle{definition}
\newtheorem{definition}[theorem]{Definition}
\newtheorem{example}[theorem]{Example}
\theoremstyle{remark}
\newtheorem{remark}[theorem]{Remark}

\newcommand{\R}{\mathbb R}
\newcommand{\F}{\mathbb F}
\newcommand{\mug}{\mu_g}

\newcommand{\Jg}{J_g}
\newcommand{\Sg}{S_g}
\newcommand{\ind}{\mathbf 1}
\newcommand{\supp}{\operatorname{supp}}

\title{\textbf{Weak Stieltjes Derivatives and Stieltjes--Sobolev Spaces:\\
A First-Order Formulation}}

\author{Francisco Javier Fern\'andez Fern\'andez\\
\small Department of Statistics, Mathematical Analysis and Optimisation\\
\small Universidade de Santiago de Compostela, Santiago de Compostela, Spain\\
\small \texttt{fjavier.fernandez@usc.es}}
\date{}

\begin{document}
\maketitle

\begin{abstract}
We develop a first-order weak \(g\)-derivative formulation for the
Stieltjes--Sobolev spaces associated with a nondecreasing left-continuous
derivator \(g\).  Starting from the strong Stieltjes fundamental theorem and
the corresponding integration-by-parts structure, we identify a weak identity
whose test pairing uses the posterior value \(\Sg\Phi\) at the atoms of the
Lebesgue--Stieltjes measure.  We prove the equivalence between this weak
derivative, a \(g\)-absolutely continuous representative, and the integral
representation defining \(W_g^{1,p}\), for \(1\le p\le\infty\).  The formulation
also makes explicit the endpoint information required when a last positive
atom has no subsequent measure mass.  We record the conjugate-exponent test
formulation, a Poincar\'e--Stieltjes estimate, and the associated \(BV\)
derivative-measure identity.  As a direct functional application, we formulate
a coercive mixed \(g\)-variational boundary-value problem, prove existence and
uniqueness by Lax--Milgram, and reconstruct the same solution explicitly as a
strong first-order flux system.  A fully solvable clock with one atom shows how
the variational solution, the strong Stieltjes solution, and the posterior
atomic orientation agree exactly.
\end{abstract}



\section{Introduction}\label{sec:introduction}

In one dimension, Sobolev regularity can be described equivalently through a
weak derivative, an absolutely continuous representative, and an integral
representation.  The purpose of this paper is to formulate the corresponding
first-order characterization in Stieltjes calculus in terms of a weak
\(g\)-derivative and to show how the resulting structure enters a genuine
\(g\)-variational problem.

Let \(g\) be nondecreasing and left-continuous on \(I=[a,b]\), and let \(\mug\)
be its Lebesgue--Stieltjes measure.  For a \(g\)-absolutely continuous
representative \(U\), we denote by \(D_gU\) the unique \(L_g^1([a,b))\)-density in the
Stieltjes fundamental theorem; this convention is stated precisely in
\cref{def:Dg}.  Thus
\[
 U(t)=U(a)+\int_{[a,t)}D_gU\,d\mug.
\]
At a jump point,
\[
 U(t^+)-U(t)=D_gU(t)\Delta g(t).
\]
Hence continuous evolution, singular-continuous evolution, plateaus, and jumps
are described within the same measure framework.  The strong Stieltjes
fundamental theorem and the associated differential calculus are well
established; see, among others, \cite{LoRo14,FriLo17,LoMa18,MaMon21}.  The
Stieltjes--Sobolev spaces used below were studied in \cite{FerToVil24}.

The weak formulation is obtained from the Stieltjes integration-by-parts
identity.  The point at which the atomic case differs visibly from the
classical one is the value selected for the second factor.  If
\[
 \Sg\varphi(t)=
 \begin{cases}
  \varphi(t^+),&\Delta g(t)>0,\\
  \varphi(t),&\Delta g(t)=0,
 \end{cases}
\]
then the identity relevant to weak differentiation is
\begin{equation}\label{eq:intro-ibp}
 U(b)\varphi(b)-U(a)\varphi(a)
 =
 \int_{[a,b)}U\,D_g\varphi\,d\mug
 +
 \int_{[a,b)}D_gU\,\Sg\varphi\,d\mug.
\end{equation}
At an atom, the product rule of
\cite[Proposition~3.9, Eq.~(3.4)]{FerMarTo22} contains the correction
\(\Delta g\,D_gU\,D_g\varphi\), while
\(\Sg\varphi(t)=\varphi(t)+\Delta g(t)D_g\varphi(t)\).  Thus the correction is
packaged into a single posterior value.  This is the orientation used
throughout the paper.

The strong ingredients underlying this formulation have several established
sources.  The fundamental theorem is available in \cite{LoRo14}; the product
rule and its jump structure are developed in \cite{FerMarTo22}; and an
integration-by-parts formula for \(g\)-absolutely continuous functions with the
explicit atomic correction appears in \cite[Lemma~2.13]{FerMarTo24}.  These
identities also fit the general Kurzweil--Stieltjes theory of
\cite{MonSlaTvr18} and the lateral \(BV\) formulation in
\cite[Appendix~B]{MarSlaTvr23}.  Our use of them is to build a single
\(L_g^p([a,b))\) weak-derivative framework and prove the corresponding global
weak--strong reconstruction.

Related weak theories provide useful points of comparison.  On time scales,
Sobolev spaces and weak reconstruction already incorporate shifted test values
on the discrete part of the scale \cite{Agarwal2006}, while the Hilger
derivative admits a Radon--Nikodym interpretation \cite{EckhardtTeschl2012}.
The one-sided measure setting of Simas and Sousa \cite{SimasSousa2025} also
uses oriented measures in a weak/energy framework, and measure differential
and distributional equations provide broader settings for regulated or
bounded-variation solutions \cite{MonteiroSatco2017}.  Here we work directly
with one arbitrary derivator \(g\), allowing plateaus, singular-continuous mass,
and atoms in the same formulation.

The main result characterizes the previously established integral space
\(W_g^{1,p}([a,b))\), \(1\le p\le\infty\), through a weak \(g\)-derivative.  Two atomic
features are made explicit.  First, the test function is evaluated at its
posterior value at a jump.  Second, a bare \(L_g^p([a,b))\)-class can lose the
posterior increment of a last positive atom when no \(\mug\)-mass follows it;
keeping the endpoint values in the global weak state removes that ambiguity.
We also derive a conjugate-exponent test formulation, a
Poincar\'e--Stieltjes estimate, and the identity between the classical \(BV\)
derivative measure and \((D_gU)\mug\).  The first-order viewpoint then leads to
a coercive mixed \(g\)-variational boundary problem with a natural posterior
forcing term.  Its Lax--Milgram solution is shown to coincide with an explicit
strong flux reconstruction.  The final worked example uses a derivator with a
continuous part and one positive atom, so that the variational identity, the
strong integral solution, and the jump rule can all be checked in closed form.

\begin{table}[ht]
\centering
\small
\caption{A first-order classical/Stieltjes dictionary.  The additional
one-sided value is needed only at atoms.}
\label{tab:dictionary}
\begin{tabularx}{\textwidth}{@{}p{0.27\textwidth}XX@{}}
\toprule
 & Classical interval & Stieltjes clock \\
\midrule
Reference measure & \(dt\) & \(d\mug\) \\
Strong representative & \(AC([a,b])\) & \(AC_g([a,b])\) \\
Derivative density & \(U'\) & \(D_gU\) \\
Fundamental formula &
\(U(t)=U(a)+\int_a^tU'\,ds\) &
\(U(t)=U(a)+\int_{[a,t)}D_gU\,d\mug\) \\
Weak second factor & \(\Phi\) & \(S_g\Phi\) at atoms \\
Boundary information & traces of the AC representative & traces of the
\(AC_g([a,b])\) representative; a last atom may hide one posterior increment from
the \(L_g^p([a,b))\)-class \\
\bottomrule
\end{tabularx}
\end{table}

The paper is organized as follows. \Cref{sec:setting} fixes the Stieltjes
setting, the continuous/jump decomposition, and the distinction between
\(L_g^p([a,b))\)-classes and strong representatives.  \Cref{sec:ibp} records the
lateral integration-by-parts formula in the form used below.
\Cref{sec:weak} introduces the weak \(g\)-derivative and proves the
weak--strong characterization.  \Cref{sec:consequences} gives the
conjugate-exponent test formulation, a Poincar\'e estimate, and the exact
terminal atomic defect.  \Cref{sec:bv} compares the formulation with \(BV\)
derivative measures.  \Cref{sec:variational} develops a coercive mixed
\(g\)-variational boundary problem, reconstructs its strong flux solution, and
closes with an exactly solvable continuous--atomic example.

\section{Stieltjes setting and strong representatives}\label{sec:setting}

\subsection{The derivator and its associated measure}
\label{subsec:derivator-measure}

Let \(I=[a,b]\), \(a<b\), and let
\[
 g:I\longrightarrow\R
\]
be nondecreasing and left-continuous.  Its finite Lebesgue--Stieltjes measure
\(\mug\) is characterized by
\begin{equation}\label{eq:measure-definition}
 \mug([s,t))=g(t)-g(s),
 \qquad a\le s\le t\le b.
\end{equation}
We write
\[
 M_g:=\mug([a,b))=g(b)-g(a)
\]
and assume \(M_g>0\) unless the degenerate case is mentioned explicitly.  For
\(1\le p\le\infty\),
\[
 L_g^p([a,b);\F):=L^p([a,b),\mug;\F),
\]
with the usual identification modulo \(\mug\)-almost-everywhere equality.  When
the scalar field is clear, we write \(L_g^p([a,b))\); throughout the
mathematical body of the paper the interval is kept explicit.  All whole-interval
integrals below are understood over \([a,b)\).

For \(t\in[a,b)\), set
\[
 \Delta g(t):=g(t^+)-g(t),
 \qquad
 \Jg:=\{t\in[a,b):\Delta g(t)>0\}.
\]
Monotonicity implies that \(\Jg\) is at most countable and
\begin{equation}\label{eq:atom-mass}
 \mug(\{t\})=\Delta g(t),
 \qquad t\in\Jg.
\end{equation}
Thus the jumps of \(g\) are exactly the atoms of \(\mug\).

We use the continuous/jump decomposition from \cite[Remark~2.1]{FerMarTo22}.
On \([a,b]\), define
\begin{equation}\label{eq:gB}
 g^B(t):=\sum_{s\in[a,t)\cap\Jg}\Delta g(s),
 \qquad
 g^C(t):=g(t)-g^B(t).
\end{equation}
Then \(g^B\) is nondecreasing and left-continuous, whereas \(g^C\) is
nondecreasing and continuous.  The associated measures satisfy
\begin{equation}\label{eq:measure-decomp}
 \mug
 =\mu_{g^C}+\mu_{g^B},
 \qquad
 \mu_{g^B}
 =\sum_{t\in\Jg}\Delta g(t)\delta_t,
\end{equation}
so \(\mu_{g^C}\) is non-atomic and \(\mu_{g^B}\) is purely atomic.  Equivalently,
for every \(f\in L_g^1([a,b))\) and \(t\in[a,b]\),
\begin{equation}\label{eq:integral-decomp}
 \int_{[a,t)}f\,d\mug
 =
 \int_{[a,t)}f\,d\mu_{g^C}
 +
 \sum_{s\in[a,t)\cap\Jg}f(s)\Delta g(s);
\end{equation}
see \cite[Lemma~2.3 and Corollary~2.4]{FerMarTo22}.  Notice that continuity of
\(g^C\) does not mean absolute continuity with respect to Lebesgue measure:
\(\mu_{g^C}\) may contain both absolutely continuous and singular-continuous
parts.  The decomposition \eqref{eq:measure-decomp} separates only the
non-atomic and purely atomic components; compare the general finite-measure
decomposition in \cite{Bartle1995}.

\subsection{Absolute continuity with respect to \texorpdfstring{\(g\)}{g}}

Before recalling absolute continuity with respect to \(g\), we fix the
pointwise Stieltjes derivative that underlies the strong density used below.
At an ordinary non-jump point where the two-sided quotient is applicable, one
sets
\begin{equation}\label{eq:pointwise-g-derivative}
 D_gU(t)
 :=
 \lim_{s\to t}\frac{U(s)-U(t)}{g(s)-g(t)},
\end{equation}
whenever the limit exists.  At a jump point the derivative is right-sided:
\begin{equation}\label{eq:pointwise-g-derivative-jump}
 D_gU(t)
 :=
 \lim_{s\to t^+}\frac{U(s)-U(t)}{g(s)-g(t)},
 \qquad t\in\Jg,
\end{equation}
and therefore, whenever the right limit of \(U\) exists,
\begin{equation}\label{eq:pointwise-g-derivative-atom}
 D_gU(t)
 =
 \frac{U(t^+)-U(t)}{\Delta g(t)},
 \qquad t\in\Jg.
\end{equation}
The complete pointwise convention also covers one-sided points and intervals
on which \(g\) is locally constant.  We do not reproduce that auxiliary
machinery here; it is given in
\cite[Definitions~3.1 and~3.7]{FerMarTo22}, with the atomic formula above in
\cite[Remark~3.8]{FerMarTo22}.  In the present paper the pointwise derivative
will be used only through its \(\mug\)-almost-everywhere relation with the
integral density and through the exact formula at atoms.

The same source gives the product rule.  If two functions are
\(g\)-differentiable at a point, then
\cite[Proposition~3.9, Eq.~(3.4)]{FerMarTo22} gives the full pointwise formula
(with the source's convention on locally constant intervals).  At points not
lying in such an interval it takes the familiar corrected form
\begin{equation}\label{eq:g-product-rule-prelim}
 D_g(UV)(t)
 =
 U(t)D_gV(t)+V(t)D_gU(t)
 +\Delta g(t)D_gU(t)D_gV(t).
\end{equation}
In particular, away from the atoms the last term vanishes, while at a jump it
records the complete product increment.  Locally constant intervals carry no
\(\mug\)-mass; consequently, after passage to \(g\)-absolutely continuous
functions and the fundamental theorem below, \eqref{eq:g-product-rule-prelim}
is exactly the \(\mug\)-almost-everywhere form needed in the sequel.

We use the standard notion of \(g\)-absolute continuity.

\begin{definition}\label{def:ACg}
A function \(U:[a,b]\to\F\) belongs to \(AC_g([a,b])\) if, for every
\(\varepsilon>0\), there exists \(\delta>0\) such that for every finite
family of pairwise disjoint half-open intervals
\(\{[a_j,b_j)\}_{j=1}^m\subset[a,b)\),
\[
 \sum_{j=1}^m\bigl(g(b_j)-g(a_j)\bigr)<\delta
 \quad\Longrightarrow\quad
 \sum_{j=1}^m|U(b_j)-U(a_j)|<\varepsilon.
\]
\end{definition}

The following established characterization fixes the strong representative
used throughout; see \cite[Theorem~5.4]{LoRo14}.

\begin{theorem}[Fundamental theorem of Stieltjes calculus]\label{thm:FTC}
For \(U:[a,b]\to\F\), the following are equivalent:
\begin{enumerate}[label=\roman*)]
\item \(U\in AC_g([a,b])\);
\item there exists \(v\in L_g^1([a,b))\) such that
\begin{equation}\label{eq:strong-rep}
 U(t)=U(a)+\int_{[a,t)}v(s)\,d\mug(s),
 \qquad t\in[a,b].
\end{equation}
\end{enumerate}
The density \(v\) is unique \(\mug\)-almost everywhere and agrees
\(\mug\)-almost everywhere with the pointwise Stieltjes derivative wherever
the latter is defined.
\end{theorem}

\begin{definition}[Strong Stieltjes derivative]\label{def:Dg}
If \(U\in AC_g([a,b])\), the unique density \(v\in L_g^1([a,b))\) in
\eqref{eq:strong-rep} is called the \emph{strong Stieltjes derivative} of
\(U\) and is denoted by \(D_gU\).  Thus, throughout this paper,
\begin{equation}\label{eq:Dg-definition}
 U(t)=U(a)+\int_{[a,t)}D_gU(s)\,d\mug(s),
 \qquad t\in[a,b].
\end{equation}
\end{definition}

At a positive-mass atom the derivative value is intrinsic to its
\(L_g^1([a,b))\)-class.  Indeed, for \(t\in\Jg\), \eqref{eq:Dg-definition} gives
\[
 U(t^+)-U(t)
 =\int_{\{t\}}D_gU\,d\mug
 =D_gU(t)\Delta g(t),
\]
and hence
\begin{equation}\label{eq:atomic-derivative}
 D_gU(t)=\frac{U(t^+)-U(t)}{\Delta g(t)}.
\end{equation}
Two representatives of the same \(L_g^1([a,b))\)-element cannot differ at \(t\),
because \(\mug(\{t\})>0\).  Thus \eqref{eq:atomic-derivative} recovers from the
jump of \(U\) the value already encoded by the derivative class.

\subsection{Stieltjes--Sobolev spaces and strong representatives}
\label{subsec:strong-Wgp}

We recall the Stieltjes--Sobolev space of
\cite[Definition~5.1]{FerToVil24}.  Definition~5.1 imposes the increment identity pointwise on a chosen
function \(u\).  Since in this paper \(L_g^p([a,b))\) denotes the usual
space modulo \(\mug\)-almost-everywhere equality, we make that representative
dependence explicit in the following equivalent restatement:
\begin{equation}\label{eq:Wgp}
 W_g^{1,p}([a,b))
 :=
 \left\{
 \begin{array}{l|l}
 u:[a,b)\to\F &
 [u]\in L_g^p([a,b)),\ \text{and there exists }
 \widetilde u\in L_g^p([a,b))\\[1mm]
 &\displaystyle
 u(t)-u(s)=\int_{[s,t)}\widetilde u(r)\,d\mug(r),
 \quad a\le s\le t<b
 \end{array}
 \right\}.
\end{equation}
The density \(\widetilde u\) is uniquely determined \(\mug\)-almost
everywhere, and the norm is
\begin{equation}\label{eq:Wgp-norm}
 \|u\|_{W_g^{1,p}([a,b))}
 :=
 \|[u]\|_{L_g^p([a,b))}
 +
 \|\widetilde u\|_{L_g^p([a,b))}.
\end{equation}
For \(I=[a,b)\), \cite[Remark~5.3]{FerToVil24} rewrites the increment
condition as an integral representation from \(a\), and
\cite[Remark~5.4]{FerToVil24} shows, on a finite interval, that the resulting
function has a \(g\)-absolutely continuous extension to \([a,b]\) whose
pointwise Stieltjes derivative agrees with \(\widetilde u\)
\(\mug\)-almost everywhere.

Because endpoint and lateral values occur repeatedly below, we keep that
closed-interval representative visible in the notation.

\begin{definition}[Strong representative realization of \(W_g^{1,p}([a,b))\)]
\label{def:strong-Wgp}
For \(1\le p\le\infty\), let
\begin{equation}\label{eq:strong-Wgp}
 \mathcal W_g^{1,p}([a,b])
 :=
 \left\{
 U\in AC_g([a,b]):
 U|_{[a,b)}\in L_g^p([a,b)),\quad
 D_gU\in L_g^p([a,b))
 \right\}.
\end{equation}
We equip this representative space with
\begin{equation}\label{eq:strong-Wgp-norm}
 \|U\|_{\mathcal W_g^{1,p}([a,b])}
 :=
 \|U\|_{L_g^p([a,b))}
 +
 \|D_gU\|_{L_g^p([a,b))},
\end{equation}
where the first term refers to the restriction to \([a,b)\).
\end{definition}

No additional Sobolev regularity is introduced by the calligraphic notation.
Rather, \(\mathcal W_g^{1,p}([a,b])\) is the closed-interval realization of
the published space in which the strong representative is retained.  More
precisely, restriction and endpoint reconstruction give the correspondence
\begin{equation}\label{eq:Wgp-representative-relation}
 \begin{aligned}
 U\in\mathcal W_g^{1,p}([a,b])
 &\ \Longrightarrow\ 
 U|_{[a,b)}\in W_g^{1,p}([a,b)),
 \qquad \widetilde u=D_gU\quad\mug\text{-a.e.},\\
 u\in W_g^{1,p}([a,b))
 &\ \Longrightarrow\ 
 U(t)=u(a)+\int_{[a,t)}\widetilde u\,d\mug,
 \qquad t\in[a,b].
 \end{aligned}
\end{equation}
In the second line the formula agrees with \(u\) on \([a,b)\) by the defining
increment identity and supplies the terminal value \(U(b)\).  Thus
\(W_g^{1,p}([a,b))\) and \(\mathcal W_g^{1,p}([a,b])\) are canonically
identified once the strong representative is made explicit.  What can be
lost after passing further to the bare interior class \([U]\in
L_g^p([a,b))\) is precisely some pointwise information; this distinction will
matter in the weak state below.

For later use, every \(U\in\mathcal W_g^{1,p}([a,b])\) satisfies
\begin{equation}\label{eq:Wgp-strong-representation}
 U(t)
 =
 U(a)+\int_{[a,t)}D_gU(s)\,d\mug(s),
 \qquad t\in[a,b],
\end{equation}
and hence, for \(a\le s\le t\le b\),
\begin{equation}\label{eq:increment-estimate}
 |U(t)-U(s)|
 \le
 \int_{[s,t)}|D_gU(r)|\,d\mug(r).
\end{equation}
Since \(\mug([a,b))<\infty\), one has
\(L_g^p([a,b))\subset L_g^1([a,b))\), and therefore
\begin{equation}\label{eq:sup-estimate}
 |U(t)|
 \le
 |U(a)|+\|D_gU\|_{L_g^1([a,b))},
 \qquad t\in[a,b].
\end{equation}

\subsection{The posterior value}
\label{subsec:lateral-value}

The half-open primitive in \eqref{eq:Dg-definition} excludes a possible atom at
the current point.  Thus, if \(t\in\Jg\),
\[
U(t)=U(a)+\int_{[a,t)}D_gU\,d\mug,
\qquad
U(t^+)=U(a)+\int_{[a,t]}D_gU\,d\mug.
\]
This motivates the posterior-value operator
\begin{equation}\label{eq:Sg}
	\Sg U(t):=
	\begin{cases}
		U(t^+),&t\in\Jg,\\
		U(t),&t\notin\Jg.
	\end{cases}
\end{equation}
At an atom,
\begin{equation}\label{eq:Sg-formula}
	\Sg U(t)
	=U(t)+\Delta g(t)D_gU(t)
	=U(a)+\int_{[a,t]}D_gU\,d\mug.
\end{equation}

At non-jump points no atomic increment is present.  Hence, by
\eqref{eq:Sg},
\[
\Sg U(t)=U(t),
\qquad t\notin\Jg.
\]
At a jump point, instead,
\[
\Sg U(t)-U(t)
=
\Delta g(t)D_gU(t),
\qquad t\in\Jg.
\]
Thus the posterior operator differs from the identity only at the atoms
of \(\mug\), where it adds precisely the atomic increment of \(U\).
The same representation yields the bound
\begin{equation}\label{eq:Sg-bound}
 \|\Sg U\|_{\infty,[a,b]}
 \le
 |U(a)|+\|D_gU\|_{L_g^1([a,b))}.
\end{equation}
The posterior operator is exactly the quantity selected by the product rule.
Indeed, at a jump,
\[
 \Phi(t)+\Delta g(t)D_g\Phi(t)=\Phi(t^+)=\Sg\Phi(t),
\]
so \eqref{eq:g-product-rule-prelim} becomes
\begin{equation}\label{eq:g-product-rule-posterior}
 D_g(U\Phi)(t)
 =U(t)D_g\Phi(t)+D_gU(t)\Sg\Phi(t),
 \qquad t\in\Jg.
\end{equation}
For \(U,\Phi\in AC_g([a,b])\), the product \(U\Phi\) is again
\(g\)-absolutely continuous; see, in particular,
\cite[Lemma~2.13]{FerMarTo24}.  Since the strong densities agree
\(\mug\)-almost everywhere with the pointwise Stieltjes derivatives, the
same identity holds \(\mug\)-almost everywhere: off \(\Jg\) one has
\(\Sg\Phi=\Phi\), while locally constant intervals are \(\mug\)-null.
Thus the posterior value used below is not an additional convention imposed on
the product rule; it is its natural lateral form.

\section{Lateral integration by parts}\label{sec:ibp}

With the posterior form of the product rule fixed in
\eqref{eq:g-product-rule-posterior}, we now record the corresponding global
integration-by-parts identity in the form needed for the weak theory.  Its
provenance is classical within the Stieltjes framework: the underlying
pointwise product rule is
\cite[Proposition~3.9, Eq.~(3.4)]{FerMarTo22}, while for
\(g\)-absolutely continuous functions the integration-by-parts formula with
the atomic correction written explicitly is
\cite[Lemma~2.13]{FerMarTo24}.  The identity below is that same formula after
absorbing the atomic correction into the posterior factor.  It is also
compatible with the general Kurzweil--Stieltjes product formula
\cite[Theorem~6.4.2]{MonSlaTvr18} and with the lateral \(BV\) formulation in
\cite[Appendix~B]{MarSlaTvr23}.

\begin{theorem}[Lateral form of Stieltjes integration by parts]\label{thm:ibp}
If \(U,\Phi\in AC_g([a,b])\), then
\begin{equation}\label{eq:ibp}
 U(b)\Phi(b)-U(a)\Phi(a)
 =
 \int_{[a,b)}U\,D_g\Phi\,d\mug
 +
 \int_{[a,b)}D_gU\,\Sg\Phi\,d\mug.
\end{equation}
Equivalently,
\begin{equation}\label{eq:ibp-symmetric}
 U(b)\Phi(b)-U(a)\Phi(a)
 =
 \int_{[a,b)}\Sg U\,D_g\Phi\,d\mug
 +
 \int_{[a,b)}D_gU\,\Phi\,d\mug.
\end{equation}
\end{theorem}

\begin{proof}
In the present notation, \cite[Lemma~2.13]{FerMarTo24} gives
\[
\begin{aligned}
 U(b)\Phi(b)-U(a)\Phi(a)
 &=\int_{[a,b)}U\,D_g\Phi\,d\mug
   +\int_{[a,b)}\Phi\,D_gU\,d\mug \\
 &\quad+\int_{[a,b)}\Delta g\,D_gU\,D_g\Phi\,d\mug.
\end{aligned}
\]
At atoms,
\(\Sg\Phi=\Phi+\Delta g\,D_g\Phi\), whereas off \(\Jg\) one has
\(\Sg\Phi=\Phi\) and the atomic correction vanishes.  Hence
\[
 \Phi D_gU+\Delta g\,D_gU\,D_g\Phi
 =D_gU\,\Sg\Phi
 \quad\mug\text{-a.e.},
\]
which yields \eqref{eq:ibp}.  Grouping instead
\(U D_g\Phi+\Delta g\,D_gU\,D_g\Phi=(\Sg U)D_g\Phi\) gives
\eqref{eq:ibp-symmetric}.
\end{proof}

\begin{remark}[Atomic check]\label{rem:atomic-check}
At a single atom \(t\in\Jg\), the product increment is
\[
\begin{aligned}
 U(t^+)\Phi(t^+)-U(t)\Phi(t)
 &=U(t)\bigl(\Phi(t^+)-\Phi(t)\bigr)\\
 &\quad+\bigl(U(t^+)-U(t)\bigr)\Phi(t^+).
\end{aligned}
\]
Using \eqref{eq:atomic-derivative}, this equals
\[
 \Delta g(t)\left[
 U(t)D_g\Phi(t)+D_gU(t)\Sg\Phi(t)
 \right],
\]
which is precisely the atomic contribution in \eqref{eq:ibp}.  If the
posterior value \(\Phi(t^+)\) were replaced by \(\Phi(t)\), the missing term
would be
\[
 (\Delta g(t))^2D_gU(t)D_g\Phi(t).
\]
Thus \eqref{eq:ibp} packages the complete jump contribution without a separate
correction sum.
\end{remark}

If \(g\) is continuous, \(\Jg=\varnothing\) and \(\Sg\Phi=\Phi\), so
\eqref{eq:ibp} has the classical integration-by-parts form with \(d\mug\) as
the reference measure.

\section{Weak differentiation and reconstruction}\label{sec:weak}

\subsection{The strong \texorpdfstring{\(W_g^{1,\infty}\)}{Wg1-infinity} test space}
\label{subsec:test-functions}

The weak formulation is obtained by transposing \eqref{eq:ibp}.  At the basic
\(L_g^1([a,b))\) level, we therefore use
\(\mathcal W_g^{1,\infty}([a,b])\) as the test space.  No additional
test-function space is required.  Indeed, by the fundamental
theorem and \cref{def:strong-Wgp},
\begin{equation}\label{eq:test-primitive-form}
 \Phi\in\mathcal W_g^{1,\infty}([a,b])
 \quad\Longleftrightarrow\quad
 \Phi(t)=c+\int_{[a,t)}\psi(s)\,d\mug(s),
 \quad
 c\in\F,\quad \psi\in L_g^\infty([a,b)),
\end{equation}
with \(c=\Phi(a)\) and \(\psi=D_g\Phi\) \(\mug\)-almost everywhere.  For the
right-to-left implication, the primitive belongs to \(AC_g([a,b])\), has
bounded derivative, and satisfies
\[
 |\Phi(t)|\le |c|+M_g\|\psi\|_{L_g^\infty([a,b))},
 \qquad t\in[a,b],
\]
so its restriction belongs to \(L_g^\infty([a,b))\).  The converse is exactly
\eqref{eq:Wgp-strong-representation} with \(p=\infty\).

For \(\Phi\in\mathcal W_g^{1,\infty}([a,b])\), one has
\[
 D_g\Phi\in L_g^\infty([a,b)),
 \qquad
 \Sg\Phi\in L_g^\infty([a,b)),
\]
the latter by \eqref{eq:Sg-bound}.  Consequently, for
\(u,v\in L_g^1([a,b))\), both pairings
\[
 \int_{[a,b)}u\,D_g\Phi\,d\mug,
 \qquad
 \int_{[a,b)}v\,\Sg\Phi\,d\mug
\]
are well defined.  Moreover, the derivative map on this test space is onto
\(L_g^\infty([a,b))\): for every \(\psi\in L_g^\infty([a,b))\), the primitive
\(
 \Phi(t)=\int_{[a,t)}\psi\,d\mug
\)
belongs to \(\mathcal W_g^{1,\infty}([a,b])\) and satisfies
\(D_g\Phi=\psi\) \(\mug\)-almost everywhere.  This supplies the
\(L_g^1([a,b))\)--\(L_g^\infty([a,b))\) separation used below.  We impose no
endpoint condition on the full test space, because the global
integration-by-parts identity retains its boundary term.

\subsection{Why endpoint data are retained}
\label{subsec:endpoint-motivation}

Before defining the weak derivative, we explain why the boundary values in
\eqref{eq:ibp} are recorded explicitly.  At the strong level,
\begin{equation}\label{eq:final-strong}
 U(b)=U(a)+\int_{[a,b)}D_gU\,d\mug,
\end{equation}
so the final value is not independent once \(U(a)\) and \(D_gU\) are known.
At the weak level, however, the interior state is only an \(L_g^1([a,b))\)-class and
therefore retains only information detected by \(\mug\).  The initial value
fixes the additive constant in the reconstruction; the final value can also
retain the posterior increment of a last positive atom that is followed by no
further measure mass.

\begin{example}[A last atom with no subsequent mass]\label{ex:one-atom}
Fix \(d\in(a,b)\) and \(\alpha>0\), and let
\[
 g(t)=
 \begin{cases}
  0,&t\le d,\\
  \alpha,&t>d.
 \end{cases}
 \qquad\text{so that}\qquad
 \mug=\alpha\delta_d.
\]
For \(c,q\in\F\), set \(U_q(t)=c+qg(t)\).  Then
\[
 U_q(d)=c,\qquad D_gU_q(d)=q,
 \qquad U_q(d^+)=U_q(b)=c+\alpha q.
\]
Equality in \(L_g^p([a,b))\) means equality \(\mug\)-almost everywhere.  Since the
whole measure is concentrated at the single point \(d\), only the value at
\(d\) contributes to an \(L_g^p([a,b))\)-norm.  Hence, for \(1\le p<\infty\),
\[
 \|U_q-c\|_{L_g^p([a,b))}^p
 =\alpha|U_q(d)-c|^p=0,
\]
and for \(p=\infty\) the essential supremum is likewise zero.  Therefore
\begin{equation}\label{eq:one-atom-same-class}
 [U_q]=[c]\quad\text{in }L_g^p([a,b)),
 \qquad U_q(a)=c,
\end{equation}
for every \(q\).  In contrast, the derivative class does distinguish \(q\):
\(\{d\}\) has positive \(\mug\)-mass, so the value
\(D_gU_q(d)=q\) is visible in \(L_g^p([a,b))\).

The point is that the measured state \([U_q]\) samples the anterior value
\(U_q(d)=c\), while there is no later \(\mug\)-mass at which the posterior
value \(U_q(d^+)=c+\alpha q\) could be sampled.  Thus the pair
\(([U_q],U_q(a))\) is identical for all \(q\), even though the derivative
changes with \(q\).  The final trace restores exactly the missing quantity,
namely
\[
 U_q(b)-U_q(d)
 =\alpha q
 =\Delta g(d)D_gU_q(d).
\]
\end{example}

The obstruction is not a terminal plateau by itself.  It is the combination
of a positive last atom and the absence of subsequent \(\mug\)-mass.  A
continuous derivator that becomes constant near \(b\) has a terminal plateau
but no analogous atomic degree of freedom.  The weak state below therefore
records the interior class together with initial and final values; the exact
circumstances under which the final value becomes redundant are identified in
\cref{prop:terminal-defect}.

\subsection{Weak Stieltjes derivative}\label{subsec:weak-derivative}

\begin{definition}[Weak Stieltjes derivative]\label{def:weak}
Let \(u,v\in L_g^1([a,b))\) and let \(u_a,u_b\in\F\).  We say that \(v\) is a weak
\(g\)-derivative of the state
\[
 ([u],u_a,u_b)
\]
if
\begin{equation}\label{eq:weak}
 \int_{[a,b)}u\,D_g\Phi\,d\mug
 +
 \int_{[a,b)}v\,\Sg\Phi\,d\mug
 =
 u_b\Phi(b)-u_a\Phi(a)
\end{equation}
for every \(\Phi\in\mathcal W_g^{1,\infty}([a,b])\).
\end{definition}

For a strong representative \(U\), the endpoint data are \(u_a=U(a)\) and
\(u_b=U(b)\).  Their explicit appearance in \eqref{eq:weak} retains at the
weak level pointwise information that can be lost when \(U\) is replaced by
its \(L_g^1([a,b))\)-class.

\subsection{Separation and injectivity lemmas}

For the separation argument only, set
\begin{equation}\label{eq:zero-endpoint-test-space}
 \mathcal W_{g,0}^{1,\infty}([a,b])
 :=
 \{\Phi\in\mathcal W_g^{1,\infty}([a,b]):\Phi(a)=\Phi(b)=0\}.
\end{equation}
Here the subscript \(0\) means that both endpoint values vanish.

\begin{lemma}\label{lem:range-zero}
\[
 D_g(\mathcal W_{g,0}^{1,\infty}([a,b]))
 =
 \left\{
 \psi\in L_g^\infty([a,b)):
 \int_{[a,b)}\psi\,d\mug=0
 \right\},
\]
where equality is understood at the level of \(L_g^\infty([a,b))\)-classes.
\end{lemma}

\begin{proof}
If \(\Phi\in\mathcal W_{g,0}^{1,\infty}([a,b])\), the fundamental theorem gives
\[
 \int_{[a,b)}D_g\Phi\,d\mug=\Phi(b)-\Phi(a)=0.
\]
Conversely, if \(\psi\in L_g^\infty([a,b))\) has zero mean, then
\[
 \Phi(t)=\int_{[a,t)}\psi\,d\mug
\]
belongs to \(\mathcal W_{g,0}^{1,\infty}([a,b])\) and satisfies
\(D_g\Phi=\psi\) \(\mug\)-almost everywhere by \cref{thm:FTC,def:Dg}.
\end{proof}

\begin{lemma}\label{lem:annihilator}
Let \(w\in L_g^1([a,b))\).  If
\[
 \int_{[a,b)}w\psi\,d\mug=0
\]
for every bounded zero-mean \(\psi\), then \(w\) is constant
\(\mug\)-almost everywhere.
\end{lemma}

\begin{proof}
Put
\[
 c=\frac1{M_g}\int_{[a,b)}w\,d\mug.
\]
For every measurable \(E\subset[a,b)\),
\[
 \psi_E:=\ind_E-\frac{\mug(E)}{M_g}
\]
has zero mean.  Hence
\[
 0=\int_{[a,b)}w\psi_E\,d\mug
  =\int_E(w-c)\,d\mug.
\]
Thus the finite signed or complex measure
\(\lambda(E):=\int_E(w-c)\,d\mug\) vanishes on every measurable set.  Its
Radon--Nikodym density with respect to \(\mug\) is \(w-c\); uniqueness of the
Radon--Nikodym density therefore gives \(w-c=0\) \(\mug\)-almost everywhere;
see, for example, \cite{Bartle1995}.
\end{proof}

The following fact records the corresponding injectivity at the level of
strong representatives.

\begin{lemma}\label{lem:injectivity}
Let \(W\in AC_g([a,b])\).  If
\[
 W=0\quad\mug\text{-a.e.},
 \qquad
 W(a)=W(b)=0,
\]
then \(W\equiv0\) on \([a,b]\) and \(D_gW=0\) \(\mug\)-almost everywhere.
\end{lemma}

\begin{proof}
Define the finite signed or complex Borel measure
\begin{equation}\label{eq:injectivity-measure}
 \nu(E):=\int_E D_gW\,d\mug,
 \qquad E\subset[a,b)\text{ measurable}.
\end{equation}
In the customary shorthand, \(\nu=(D_gW)\mug\).  Since \(D_gW\in L_g^1([a,b))\),
this measure is finite and \(\nu\ll\mug\).  Because \(W(a)=0\), the
fundamental theorem gives
\begin{equation}\label{eq:injectivity-cumulative}
 W(t)=\nu([a,t)),
 \qquad
 W(t)-W(s)=\nu([s,t)),
 \qquad a\le s\le t\le b.
\end{equation}

Here and in \cref{prop:terminal-defect}, \(\supp\mug\) is
understood in the relative topology of \([a,b]\), after regarding \(\mug\)
as a measure on \([a,b]\) with \(\mug(\{b\})=0\).

We first prove that \(W\) vanishes on \(\supp\mug\).  If
\(t\in\supp\mug\cap\Jg\), then \(\mug(\{t\})>0\).  A point of positive
measure cannot belong to the exceptional null set on which the assertion
\(W=0\) may fail, so \(W(t)=0\).

Now let \(t\in\supp\mug\setminus\Jg\).  Since \(\nu\ll\mug\),
\(\nu(\{t\})=0\).  The cumulative function
\(r\mapsto\nu([a,r))\) is always left-continuous, and its right jump at \(t\)
is exactly \(\nu(\{t\})\); hence it is continuous at this non-atomic point.
If \(W(t)\ne0\), continuity would give a neighborhood \(I\) of \(t\) and a
number \(\eta>0\) such that \(|W|\ge\eta\) on \(I\cap[a,b]\).  Because
\(t\in\supp\mug\), every such neighborhood has positive \(\mug\)-measure.
This contradicts \(W=0\) \(\mug\)-almost everywhere.  Therefore
\(W=0\) on \(\supp\mug\).

It remains to consider the complement of the support.  Each connected
component \(J\) of \([a,b]\setminus\supp\mug\) carries no \(\mug\)-mass.
Thus, for \(s<t\) in the same component,
\eqref{eq:injectivity-cumulative} gives \(W(t)-W(s)=0\), so \(W\) is constant
on \(J\).  If the component meets \(a\), this constant is \(W(a)=0\).  Otherwise
its right endpoint is either a point of \(\supp\mug\) or \(b\).  In the first
case left continuity of \(W\) identifies the constant with the value at that
support point, which is zero; in the second case it is \(W(b)=0\).  Hence
\(W\equiv0\) on all of \([a,b]\).

Finally, \eqref{eq:injectivity-cumulative} now gives
\(\nu([s,t))=0\) for every half-open interval.  Such intervals generate the
Borel \(\sigma\)-algebra, so the finite measure \(\nu\) is zero.  By uniqueness
of its Radon--Nikodym density with respect to \(\mug\),
\(D_gW=0\) \(\mug\)-almost everywhere; cf. \cite{Bartle1995}.
\end{proof}

\subsection{The weak--strong equivalence}

The preceding lemmas provide the separation needed to answer the central
first-order question.  Suppose that only an interior \(L_g^1([a,b))\)-class, two
candidate endpoint values, and a candidate derivative density are given.  We
now characterize exactly when these data arise from one \(g\)-absolutely
continuous representative.  The theorem shows that the global weak identity,
the strong representative, and the explicit primitive reconstruction are
three equivalent descriptions of the same state.

\begin{theorem}[First-order weak characterization]\label{thm:weak-strong}
Let \(u,v\in L_g^1([a,b))\) and \(u_a,u_b\in\F\).  The following are equivalent:
\begin{enumerate}[label=\roman*)]
\item the quadruple \(([u],u_a,u_b;v)\) satisfies
\begin{equation}\label{eq:weak-main-displayed}
 \int_{[a,b)}u\,D_g\Phi\,d\mug
 +\int_{[a,b)}v\,\Sg\Phi\,d\mug
 =u_b\Phi(b)-u_a\Phi(a)
\end{equation}
for every \(\Phi\in\mathcal W_g^{1,\infty}([a,b])\); equivalently, \eqref{eq:weak} holds;
\item there exists \(U\in AC_g([a,b])\) such that
\[
 U=u\quad\mug\text{-a.e.},\qquad
 U(a)=u_a,\quad U(b)=u_b,\quad
 D_gU=v\quad\mug\text{-a.e.};
\]
\item the primitive
\begin{equation}\label{eq:reconstruction}
 U(t):=u_a+\int_{[a,t)}v\,d\mug
\end{equation}
represents \([u]\) and satisfies \(U(b)=u_b\).
\end{enumerate}
When these conditions hold, both \(U\) and \(v\) are unique for the prescribed
endpoint data.
\end{theorem}

\begin{proof}
\begin{enumerate}[leftmargin=0pt,label={},itemindent=0pt,itemsep=.65\baselineskip]
\item \emph{(i) \(\Rightarrow\) (iii).}
Assume (i) and define \(V\) by \eqref{eq:reconstruction}.  Then
\(V\in AC_g([a,b])\), \(V(a)=u_a\), and \(D_gV=v\)
\(\mug\)-almost everywhere.  Integration by parts gives
\[
 \int_{[a,b)}V D_g\Phi\,d\mug
 +\int_{[a,b)}v\Sg\Phi\,d\mug
 =V(b)\Phi(b)-u_a\Phi(a).
\]
Subtracting this identity from \eqref{eq:weak-main-displayed},
\begin{equation}\label{eq:orthogonality}
 \int_{[a,b)}(u-V)D_g\Phi\,d\mug
 =(u_b-V(b))\Phi(b),
 \qquad \Phi\in\mathcal W_g^{1,\infty}([a,b]).
\end{equation}
For \(\Phi\in\mathcal W_{g,0}^{1,\infty}([a,b])\),
\cref{lem:range-zero,lem:annihilator} imply that
\[
 u-V=c\qquad\mug\text{-a.e.}
\]
for some \(c\in\F\).  Taking \(\Phi\equiv1\) in
\eqref{eq:orthogonality} yields \(u_b=V(b)\).  Hence
\[
 c\int_{[a,b)}D_g\Phi\,d\mug=0
 \qquad \forall\Phi\in\mathcal W_g^{1,\infty}([a,b]).
\]
For \(\Phi(t)=g(t)-g(a)=\int_{[a,t)}1\,d\mug\), one has
\(D_g\Phi=1\) \(\mug\)-almost everywhere, so \(cM_g=0\).  Since \(M_g>0\),
\(c=0\).  Thus \(V\) represents \([u]\) and \(V(b)=u_b\), which is (iii).

\item \emph{(iii) \(\Rightarrow\) (ii).}
By \cref{thm:FTC,def:Dg}, the primitive in \eqref{eq:reconstruction} belongs
to \(AC_g([a,b])\), satisfies \(U(a)=u_a\), and has
\(D_gU=v\) \(\mug\)-almost everywhere.  The hypotheses in (iii) provide
\(U=u\) \(\mug\)-almost everywhere and \(U(b)=u_b\), so (ii) follows.

\item \emph{(ii) \(\Rightarrow\) (i).}
Apply \cref{thm:ibp} to the representative \(U\) from (ii) and any
\(\Phi\in\mathcal W_g^{1,\infty}([a,b])\).  Replacing \(U\) and \(D_gU\) by their
\(\mug\)-almost-everywhere representatives \(u\) and \(v\) gives
\eqref{eq:weak-main-displayed}.

\item \emph{Uniqueness.}
Suppose \(v_1,v_2\) satisfy the weak identity for the same state, and let
\(U_1,U_2\) be the corresponding reconstructed representatives.  Then
\(W:=U_1-U_2\) satisfies
\[
 W=0\quad\mug\text{-a.e.},\qquad W(a)=W(b)=0.
\]
By \cref{lem:injectivity}, \(W\equiv0\) and
\(D_gW=v_1-v_2=0\) \(\mug\)-almost everywhere.  Hence both the strong
representative and the derivative density are unique.
\end{enumerate}
\end{proof}

\begin{remark}[Why the final trace occurs in the weak state]
\label{rem:why-final-trace}
The coordinate \(u_b\) is not an additional independent datum in the strong
theory: once \(u_a\) and \(D_gU\) are known, \eqref{eq:final-strong}
determines \(U(b)\).  Its role is different before reconstruction.  In the
one-atom configuration of \cref{ex:one-atom}, all functions \(U_q\) determine
the same interior class \([c]\) and the same initial value \(c\), while their
derivatives have distinct values \(q\) on a set of positive \(\mug\)-measure.
The missing information is not sampled later because there is no
\(\mug\)-mass after the atom.  The final value
\(u_b=c+\alpha q\) records precisely that posterior increment and therefore
selects the correct derivative density.  \Cref{prop:terminal-defect} shows
that this last-atom configuration is the only obstruction to recovering the
derivative from \(([u],u_a)\) alone.
\end{remark}

\begin{corollary}[Weak characterization of the strong realization of \(W_g^{1,p}([a,b))\)]\label{cor:Wgp}
Let \(1\le p\le\infty\).  A state \(([u],u_a,u_b)\), with
\(u\in L_g^p([a,b))\), is represented by some
\(U\in\mathcal W_g^{1,p}([a,b])\) if and only if there exists
\(v\in L_g^p([a,b))\) such that
\begin{equation}\label{eq:weak-Wgp-displayed}
 \int_{[a,b)}u\,D_g\Phi\,d\mug
 +\int_{[a,b)}v\,\Sg\Phi\,d\mug
 =u_b\Phi(b)-u_a\Phi(a)
\end{equation}
for every \(\Phi\in\mathcal W_g^{1,\infty}([a,b])\); this is exactly the weak identity
\eqref{eq:weak}.  In that case,
\[
 U(t)=u_a+\int_{[a,t)}v\,d\mug,
 \qquad D_gU=v\quad\mug\text{-a.e.},
\]
and the representative is unique for the prescribed endpoint data.
\end{corollary}

\begin{proof} $ $
\begin{enumerate}[leftmargin=0pt,label={},itemindent=0pt,itemsep=.65\baselineskip]
\item \emph{Strong \(\Rightarrow\) weak.}
If \(U\in\mathcal W_g^{1,p}([a,b])\), take \(u=U\) and \(v=D_gU\) as
\(L_g^p([a,b))\)-classes and \(u_a=U(a)\), \(u_b=U(b)\).  Then
\cref{thm:ibp} gives \eqref{eq:weak-Wgp-displayed} for every
\(\Phi\in\mathcal W_g^{1,\infty}([a,b])\).

\item \emph{Weak \(\Rightarrow\) strong.}
Assume \eqref{eq:weak-Wgp-displayed} for some \(v\in L_g^p([a,b))\).  Since \(\mug\)
is finite, \(L_g^p([a,b))\subset L_g^1([a,b))\), so \cref{thm:weak-strong} reconstructs the
unique \(U\in AC_g([a,b])\) with the prescribed traces,
\(U=u\) \(\mug\)-almost everywhere, and \(D_gU=v\)
\(\mug\)-almost everywhere.  The class \(u\) belongs to \(L_g^p([a,b))\) by
hypothesis and \(D_gU=v\in L_g^p([a,b))\); hence
\(U\in\mathcal W_g^{1,p}([a,b])\) by \cref{def:strong-Wgp}.
\end{enumerate}
\end{proof}

\section{Conjugate-exponent tests, estimates, and terminal information}
\label{sec:consequences}

\subsection{The \texorpdfstring{\(L_g^p([a,b))\)--\(L_g^{p'}([a,b))\)}{Lp-Lp-prime} test formulation}

Let \(p'\) denote the conjugate exponent of \(p\).  The terminology
``conjugate exponent'' here refers only to the H\"older pairing used in the
weak identity.  It is \emph{not} an identification of the Banach dual of
\(\mathcal W_g^{1,p}([a,b])\) with \(\mathcal W_g^{1,p'}([a,b])\).  Such an identification is
not the standard Sobolev duality even in the classical case: after the graph
embedding \(U\mapsto(U,D_gU)\), continuous functionals are naturally
represented by pairs of conjugate-exponent functions modulo the annihilator
of the graph (and, at the endpoint exponents, the Banach dual is still more
subtle).  In the Hilbert case \(p=2\), Riesz representation gives an abstract
identification with the Hilbert space after a choice of inner product, but
that is a different statement from the integral test pairing below.  What we
need here is only that \(\mathcal W_g^{1,p'}([a,b])\) is a sufficiently regular and
H\"older-compatible test space.

\begin{lemma}[Trace and posterior-value estimates]\label{lem:trace-shift}
Let \(1\le p\le\infty\), let \(p'\) be its conjugate exponent, and let
\(\Phi\in\mathcal W_g^{1,p'}([a,b])\).  With the conventions
\(M_g^{-1/\infty}=1\) and \(M_g^{1/\infty}=1\), one has
\begin{equation}\label{eq:trace-control}
 |\Phi(a)|
 \le
 M_g^{-1/p'}\|\Phi\|_{L_g^{p'}([a,b))}
 +M_g^{1/p}\|D_g\Phi\|_{L_g^{p'}([a,b))},
\end{equation}
\begin{equation}\label{eq:trace-control-b}
 |\Phi(b)|
 \le
 M_g^{-1/p'}\|\Phi\|_{L_g^{p'}([a,b))}
 +2M_g^{1/p}\|D_g\Phi\|_{L_g^{p'}([a,b))},
\end{equation}
and
\begin{equation}\label{eq:Sg-Lp-bound}
 \|\Sg\Phi\|_{L_g^{p'}([a,b))}
 \le
 \|\Phi\|_{L_g^{p'}([a,b))}+M_g\|D_g\Phi\|_{L_g^{p'}([a,b))}.
\end{equation}
In particular, both endpoint traces and the map
\(\Phi\mapsto\Sg\Phi\) are continuous on
\(\mathcal W_g^{1,p'}([a,b])\).
\end{lemma}

\begin{proof}
For every \(t\in[a,b]\), the integral representation gives
\[
 |\Phi(a)|
 \le |\Phi(t)|+\|D_g\Phi\|_{L_g^1([a,b))}.
\]
If \(p'<\infty\), taking the \(L_g^{p'}([a,b))\)-norm of this inequality as a
function of \(t\) and dividing by \(M_g^{1/p'}\) yields
\[
 |\Phi(a)|
 \le M_g^{-1/p'}\|\Phi\|_{L_g^{p'}([a,b))}+\|D_g\Phi\|_{L_g^1([a,b))}.
\]
For \(p'=\infty\), the same estimate follows by taking the essential
supremum.  H\"older's inequality on the finite measure space gives
\[
 \|D_g\Phi\|_{L_g^1([a,b))}
 \le M_g^{1/p}\|D_g\Phi\|_{L_g^{p'}([a,b))},
\]
including the endpoint cases \(p=1\) and \(p=\infty\).  This proves
\eqref{eq:trace-control}.  Since
\(|\Phi(b)|\le|\Phi(a)|+\|D_g\Phi\|_{L_g^1([a,b))}\),
\eqref{eq:trace-control-b} follows.

Finally, at atoms \eqref{eq:Sg-formula} gives
\(\Sg\Phi-\Phi=\Delta g\,D_g\Phi\), while off \(\Jg\) the difference is
zero.  Thus, after extending \(\Delta g\) by zero off \(\Jg\),
\[
 \Sg\Phi=\Phi+\Delta g\,D_g\Phi
 \quad\mug\text{-a.e.}.
\]  Because
\(0\le\Delta g\le M_g\), the triangle inequality in \(L_g^{p'}([a,b))\) yields
\eqref{eq:Sg-Lp-bound}.  The three displayed estimates prove the asserted
continuity.
\end{proof}

The next theorem should therefore be read as a change of the test class, not
as a theorem identifying topological dual spaces.

\begin{theorem}[Weak characterization with conjugate-exponent tests]
\label{thm:natural-duality}
Let \(1\le p\le\infty\), \(u,v\in L_g^p([a,b))\), and \(u_a,u_b\in\F\).  The
following are equivalent:
\begin{enumerate}[label=\roman*)]
\item for every \(\Phi\in\mathcal W_g^{1,p'}([a,b])\),
\begin{equation}\label{eq:weak-natural}
 \int_{[a,b)}uD_g\Phi\,d\mug
 +\int_{[a,b)}v\Sg\Phi\,d\mug
 =u_b\Phi(b)-u_a\Phi(a);
\end{equation}
\item the primitive
\[
 U(t)=u_a+\int_{[a,t)}v\,d\mug
\]
belongs to \(\mathcal W_g^{1,p}([a,b])\), represents \([u]\), and satisfies
\(U(b)=u_b\).
\end{enumerate}
\end{theorem}

\begin{proof}$ $
\begin{enumerate}[leftmargin=0pt,label={},itemindent=0pt,itemsep=.65\baselineskip]
\item \emph{(i) \(\Rightarrow\) (ii).}
Because \(\mug([a,b))<\infty\),
\[
 \mathcal W_g^{1,\infty}([a,b])
 \subset \mathcal W_g^{1,p'}([a,b]),
\]
with equality when \(p'=\infty\).  Hence \eqref{eq:weak-natural} holds in
particular for every basic test
\(\Phi\in\mathcal W_g^{1,\infty}([a,b])\).  The weak--strong theorem (or
\cref{cor:Wgp}) then reconstructs the stated
\(U\in\mathcal W_g^{1,p}([a,b])\).

\item \emph{(ii) \(\Rightarrow\) (i).}
Let \(U\) be the representative from (ii).  For
\(\Phi\in\mathcal W_g^{1,p'}([a,b])\), \cref{lem:trace-shift} and H\"older's
inequality ensure that all terms in \eqref{eq:weak-natural} are finite.
Applying \cref{thm:ibp} to \(U\) and \(\Phi\), and replacing \(U\) and
\(D_gU\) by the classes \(u\) and \(v\), gives
\eqref{eq:weak-natural}.
\end{enumerate}
\end{proof}

\subsection{Poincar\'e--Stieltjes estimate}

\begin{theorem}\label{thm:poincare}
Let \(U\in\mathcal W_g^{1,p}([a,b])\), \(1\le p\le\infty\).
If \(U(a)=0\), then
\begin{equation}\label{eq:poincare}
 \|U\|_{L_g^p([a,b))}\le M_g\|D_gU\|_{L_g^p([a,b))}.
\end{equation}
If \(1<p<\infty\), the sharper bound
\begin{equation}\label{eq:poincare-sharp}
 \|U\|_{L_g^p([a,b))}^p
 \le
 \|D_gU\|_{L_g^p([a,b))}^p
 \int_{[a,b)}
 \mug([a,t))^{p-1}\,d\mug(t)
\end{equation}
also holds.
\end{theorem}

\begin{proof}
If \(U(a)=0\), then
\[
 U(t)=\int_{[a,t)}D_gU\,d\mug.
\]
H\"older's inequality gives
\[
 |U(t)|
 \le
 \mug([a,t))^{1/p'}\|D_gU\|_{L_g^p([a,b))}
 \le
 M_g^{1/p'}\|D_gU\|_{L_g^p([a,b))}.
\]
Taking the \(L_g^p([a,b))\)-norm gives \eqref{eq:poincare}; raising the first
pointwise estimate to the power \(p\) before integrating gives
\eqref{eq:poincare-sharp}.
\end{proof}

The integral in \eqref{eq:poincare-sharp} should be kept in this form when
\(\mug\) has atoms.  The continuous-measure identity \(M_g^p/p\) is not valid
without modification in the presence of jumps.

\subsection{When is the final trace actually needed?}

The one-atom example shows that the final endpoint can restore information
lost by the measured class.  The following statement isolates exactly that
obstruction.  For \(v\in L_g^1([a,b))\), put
\[
 (Kv)(t):=\int_{[a,t)}v\,d\mug.
\]

\begin{proposition}[Exact terminal atomic defect]\label{prop:terminal-defect}
Assume \(M_g>0\).
\begin{enumerate}[label=\roman*)]
\item If there is no \(d\in[a,b)\) such that
\[
 \mug(\{d\})>0,
 \qquad
 \mug((d,b))=0,
\]
then
\[
 [Kv]=0\text{ in }L_g^1([a,b))
 \quad\Longrightarrow\quad
 v=0\quad\mug\text{-a.e.}
\]
\item If such a point \(d\) exists, it is unique and
\[
 \ker\bigl(v\mapsto[Kv]\bigr)
 =
 \operatorname{span}\{\ind_{\{d\}}\}
\]
inside \(L_g^1([a,b))\).
\end{enumerate}
\end{proposition}

\begin{proof}
Let \(\nu=v\mug\) and \(W(t)=\nu([a,t))\).  Assume throughout that
\([Kv]=0\), equivalently \(W=0\) \(\mug\)-almost everywhere.  The argument in
\cref{lem:injectivity} shows that \(W\) vanishes at every point of
\(\supp\mug\): at atoms this follows directly from positive singleton mass,
and at non-atomic support points from continuity of the cumulative function
and the positive measure of every neighborhood.

On each connected component of \([a,b]\setminus\supp\mug\), the function
\(W\) is constant.  Every component whose right endpoint lies in
\(\supp\mug\) has constant zero by left continuity.  Thus the only possible
place where a nonzero posterior constant can survive is a final component
following a last support point \(d\).

\begin{enumerate}[leftmargin=0pt,label={},itemindent=0pt,itemsep=.65\baselineskip]
\item \emph{Proof of (i).}
If no \(d\) satisfies
\(\mug(\{d\})>0\) and \(\mug((d,b))=0\), then no nonzero posterior constant
can survive on a final mass-free interval.  Hence \(W\equiv0\) on \([a,b]\).
Therefore \(\nu([s,t))=W(t)-W(s)=0\) for every half-open interval.  These
intervals generate the Borel \(\sigma\)-algebra, so \(\nu=0\), and uniqueness
of the Radon--Nikodym density yields \(v=0\) \(\mug\)-almost everywhere.

\item \emph{Proof of (ii).}
Suppose such a point \(d\) exists.  It is unique, because two distinct points
with positive mass cannot both have zero measure strictly to their right.  The
preceding support argument applied on \([a,d]\) shows that \(\nu\) vanishes
away from \(\{d\}\).  Hence \(v=0\) \(\mug\)-almost everywhere on
\([a,b)\setminus\{d\}\), so
\[
 v=c\ind_{\{d\}}
 \quad\text{in }L_g^1([a,b))
\]
for some \(c\in\F\).  This proves
\(\ker K\subset\operatorname{span}\{\ind_{\{d\}}\}\).
Conversely, if \(v=c\ind_{\{d\}}\), then \((Kv)(t)=0\) for \(t\le d\) and
\((Kv)(t)=c\mug(\{d\})\) for \(t>d\).  The latter region has zero
\(\mug\)-measure by assumption, so \([Kv]=0\) in \(L_g^1([a,b))\).  Thus the reverse
inclusion holds and the kernel is exactly the stated one-dimensional space.
\end{enumerate}
\end{proof}

\begin{remark}\label{rem:reduced-state}
If the terminal atomic configuration in \cref{prop:terminal-defect} is
excluded, the final trace is redundant for uniqueness of the derivative and
can be recovered from \eqref{eq:final-strong}.  In the unrestricted class of
derivators, retaining \(u_b\) is a uniform way of avoiding a case distinction.
It should not be interpreted as an independent strong datum.
\end{remark}

\section{Comparison with \texorpdfstring{\(BV\)}{BV} and classical distributional derivatives}\label{sec:bv}

A second way to view the first-order theory is through derivative measures. This comparison also makes clear that Stieltjes calculus does not create singular measures unavailable to classical distribution theory; rather, it fixes the reference measure and its lateral conventions in advance.

If \(U\in BV([a,b];\F)\) is taken in its left-continuous representative, let \(DU\) denote the finite signed or complex measure on \([a,b)\) characterized by
\begin{equation}\label{eq:BVmeasure}
 DU([s,t))=U(t)-U(s).
\end{equation}
Its restriction to \((a,b)\) is the ordinary distributional derivative. The convention on \([a,b)\) additionally records a possible right jump at \(a\).

\begin{theorem}[Derivative measure]\label{thm:measure-identity}
If \(U\in AC_g([a,b])\), then \(U\in BV([a,b])\) and
\begin{equation}\label{eq:DU}
 DU=(D_gU)\mug.
\end{equation}
Consequently, if the non-atomic component has Lebesgue decomposition
\[
 \mu_{g^C}=\rho_g\,dt+\mu_{g^C}^{\mathrm{sc}},
\]
then
\[
 \mug
 =\rho_g\,dt+\mu_{g^C}^{\mathrm{sc}}+\mu_{g^B},
\]
and
\begin{equation}\label{eq:DU-decomp}
 DU=(D_gU)\rho_g\,dt
 +(D_gU)\mu_{g^C}^{\mathrm{sc}}
 +\sum_{t\in\Jg}\Delta g(t)D_gU(t)\delta_t.
\end{equation}
\end{theorem}

\begin{proof}
For every partition \(a=t_0<\cdots<t_N=b\),
\[
 \sum_{j=1}^N|U(t_j)-U(t_{j-1})|
 \le
 \int_{[a,b)}|D_gU|\,d\mug,
\]
so \(U\in BV([a,b])\). Let \(\nu=(D_gU)\mug\). By the fundamental theorem,
\[
 \nu([s,t))=U(t)-U(s)=DU([s,t)).
\]
The half-open intervals generate the Borel \(\sigma\)-algebra, so the two finite measures coincide. Formula \eqref{eq:DU-decomp} follows from the decomposition of \(\mug\).
\end{proof}

\begin{corollary}\label{cor:BV}
Let \(U:[a,b]\to\F\) be left-continuous and of bounded variation. Then
\[
 U\in AC_g([a,b])
 \quad\Longleftrightarrow\quad
 DU\ll\mug.
\]
When these conditions hold,
\[
 D_gU=\frac{dDU}{d\mug}
 \qquad\mug\text{-a.e.}
\]
\end{corollary}

\begin{proof}
The direct implication is \cref{thm:measure-identity}. Conversely, if \(DU=v\mug\), evaluating on \([a,t)\) gives
\[
 U(t)-U(a)=\int_{[a,t)}v\,d\mug,
\]
so the fundamental theorem yields \(U\in AC_g([a,b])\) and \(D_gU=v\).
\end{proof}

\begin{remark}[Classical distributions see the same measure]\label{rem:classical}
For a mixed clock
\[
 \mug=dt+\alpha\delta_{t_0},
\]
a Stieltjes equation \(D_gU=F(t,U)\) corresponds formally to
\[
 DU=F(t,U)\,dt+\alpha F(t_0,U(t_0))\delta_{t_0}.
\]
The formula is rigorous once the coefficient of the Dirac mass is evaluated at the prescribed pre-jump value. Classical distribution theory can therefore describe the same solution. The convenience of the Stieltjes formulation is that the measure, the atomic orientation, and the integral reconstruction are part of a single reference clock rather than being supplied separately after the equation is written.
\end{remark}

This \(BV\) viewpoint also explains why the weak theory resembles the classical one. In the atomless case, the only difference from ordinary Sobolev calculus is the underlying measure. At atoms, the derivative measure still has the expected Dirac masses, but the integration-by-parts pairing must choose compatible one-sided values.

\section{A coercive \texorpdfstring{\(g\)}{g}-variational boundary-value problem}
\label{sec:variational}

The preceding first-order theory leads naturally to a coercive boundary-value
problem.  The purpose of this section is not to introduce a general new theory
of Stieltjes boundary-value equations.  Such problems have been studied under
periodic and more general boundary conditions by different methods; see, for
example, \cite{MarraffaSatcoBVP2025,MarraffaSatcoSecond2025}.  Rather, we use a
simple coercive problem to show how the lateral weak formulation developed
above gives rise to a Lax--Milgram solution and how that solution can
subsequently be reconstructed by the strong Stieltjes integral calculus.  The
role of the posterior operator will emerge directly from the lateral
integration-by-parts formula.

\subsection{Variational formulation and well-posedness}

For clarity, in this section we work over the real scalar field.  Let
\(m\ge1\), and consider functions \(U:[a,b]\to\R^m\).  
We use \(AC_g([a,b];\R^m):=(AC_g([a,b]))^m\) for the corresponding
componentwise vector-valued \(g\)-absolutely continuous space.
For vector-valued functions we set
\[
\mathcal W_g^{1,2}([a,b];\R^m)
:=
\bigl(\mathcal W_g^{1,2}([a,b])\bigr)^m,
\]
with the Stieltjes derivative understood componentwise, and use the strong
Sobolev norm
\begin{equation}\label{eq:vector-W12-norm}
	\|U\|_{\mathcal W_g^{1,2}([a,b];\R^m)}
	:=
	\|U\|_{L_g^2([a,b);\R^m)}
	+
	\|D_gU\|_{L_g^2([a,b);\R^m)}.
\end{equation}

We first identify the natural energy space.  Since the problem considered
below carries an essential boundary condition at the left endpoint, define
\begin{equation}\label{eq:variational-space}
	V_{g,a}^m([a,b])
	:=
	\left\{
	U\in\mathcal W_g^{1,2}([a,b];\R^m):
	U(a)=0
	\right\},
\end{equation}
and endow it with
\begin{equation}\label{eq:variational-space-norm}
	\|U\|_{V_{g,a}^m([a,b])}
	:=
	\|D_gU\|_{L_g^2([a,b);\R^m)}.
\end{equation}
The condition \(U(a)=0\) removes the additive constant in the Stieltjes
primitive.  Consequently, the derivative alone determines the strong
representative and therefore defines a genuine norm on
\(V_{g,a}^m([a,b])\).

\begin{lemma}[Derivative representation of the energy space]
	\label{lem:variational-derivative-isomorphism}
	The map
	\[
	D_g:
	V_{g,a}^m([a,b])
	\longrightarrow
	L_g^2([a,b);\R^m)
	\]
	is an isometric isomorphism.  Its inverse is the Stieltjes primitive operator
	\[
	Z\longmapsto U_Z,
	\qquad
	U_Z(t)
	:=
	\int_{[a,t)}Z(s)\,d\mug(s),
	\qquad t\in[a,b].
	\]
	Consequently, \(V_{g,a}^m([a,b])\), endowed with
	\eqref{eq:variational-space-norm}, is a Hilbert space.
\end{lemma}

\begin{proof}
	Let \(U\in V_{g,a}^m([a,b])\).  Since \(U(a)=0\), the fundamental theorem of
	Stieltjes calculus gives
	\[
	U(t)
	=
	\int_{[a,t)}D_gU(s)\,d\mug(s),
	\qquad t\in[a,b].
	\]
	Hence \(D_gU=0\) in \(L_g^2([a,b);\R^m)\) implies \(U=0\), so
	\eqref{eq:variational-space-norm} is indeed a norm.
	
	Conversely, let \(Z\in L_g^2([a,b);\R^m)\).  Since
	\(\mug([a,b))=M_g<\infty\), one has
	\(Z\in L_g^1([a,b);\R^m)\).  Thus
	\[
	U_Z(t)
	=
	\int_{[a,t)}Z(s)\,d\mug(s)
	\]
	defines a function in \((AC_g([a,b]))^m\) satisfying \(U_Z(a)=0\) and
	\[
	D_gU_Z=Z
	\qquad \mug\text{-a.e. on }[a,b).
	\]
	Moreover,
	\[
	|U_Z(t)|
	\le
	\int_{[a,b)}|Z(s)|\,d\mug(s)
	\le
	M_g^{1/2}
	\|Z\|_{L_g^2([a,b);\R^m)},
	\]
	and therefore
	\[
	\|U_Z\|_{L_g^2([a,b);\R^m)}
	\le
	M_g
	\|Z\|_{L_g^2([a,b);\R^m)}.
	\]
	Hence
	\(U_Z\in\mathcal W_g^{1,2}([a,b];\R^m)\), and so
	\(U_Z\in V_{g,a}^m([a,b])\).  Thus \(D_g\) is onto and, by
	\eqref{eq:variational-space-norm}, isometric.  Since
	\(L_g^2([a,b);\R^m)\) is a Hilbert space, the same is true of
	\(V_{g,a}^m([a,b])\).
\end{proof}

The energy norm controls both the anterior state \(U\) and the posterior
state \(\Sg U\).  The latter estimate is particularly important because the
posterior value is the quantity that will appear in the load functional.

\begin{lemma}[Control of the anterior and posterior states]
	\label{lem:variational-state-estimates}
	For every \(U\in V_{g,a}^m([a,b])\),
	\begin{equation}\label{eq:variational-poincare}
		\|U\|_{L_g^2([a,b);\R^m)}
		\le
		M_g
		\|U\|_{V_{g,a}^m([a,b])},
	\end{equation}
	and
	\begin{equation}\label{eq:variational-Sg-estimate}
		\|\Sg U\|_{L_g^2([a,b);\R^m)}
		\le
		M_g
		\|U\|_{V_{g,a}^m([a,b])}.
	\end{equation}
\end{lemma}

\begin{proof}
	The first estimate is the Poincar\'e--Stieltjes inequality
	\eqref{eq:poincare}, applied componentwise.
	
	For the posterior value, the condition \(U(a)=0\) and the strong
	representation give
	\[
	\Sg U(t)
	=
	\int_{[a,t]}D_gU(s)\,d\mug(s),
	\qquad t\in[a,b),
	\]
	where the inclusion of the endpoint changes the integral precisely when
	\(t\) is an atom of \(\mug\).  Consequently,
	\[
	|\Sg U(t)|
	\le
	\int_{[a,b)}|D_gU(s)|\,d\mug(s)
	\le
	M_g^{1/2}
	\|D_gU\|_{L_g^2([a,b);\R^m)}.
	\]
	Taking the \(L_g^2([a,b);\R^m)\)-norm with respect to \(t\) yields
	\[
	\|\Sg U\|_{L_g^2([a,b);\R^m)}
	\le
	M_g
	\|D_gU\|_{L_g^2([a,b);\R^m)},
	\]
	which is \eqref{eq:variational-Sg-estimate}.
\end{proof}

In particular, on the subspace \(V_{g,a}^m([a,b])\), the energy norm and
the full strong Sobolev norm are equivalent.  Indeed,
\begin{equation}\label{eq:energy-sobolev-norm-equivalence}
	\|U\|_{V_{g,a}^m([a,b])}
	\le
	\|U\|_{\mathcal W_g^{1,2}([a,b];\R^m)}
	\le
	(1+M_g)
	\|U\|_{V_{g,a}^m([a,b])}.
\end{equation}
Thus working with the derivative norm in the variational formulation does not
lose control of the state itself.

We now introduce the coefficient appearing in the constitutive relation.  Let
\[
A:[a,b)\longrightarrow\R^{m\times m}
\]
be measurable and assume that there exist constants
\(0<\lambda\le\Lambda<\infty\) such that, for \(\mug\)-almost every
\(t\in[a,b)\),
\begin{equation}\label{eq:A-ellipticity}
	\lambda|\xi|^2
	\le
	\xi^{\mathsf T}A(t)\xi,
	\qquad
	|A(t)\xi|
	\le
	\Lambda|\xi|,
	\qquad
	\xi\in\R^m.
\end{equation}
The first inequality provides coercivity, whereas the second gives
boundedness of the associated bilinear form.  No symmetry of \(A(t)\) is
assumed.

Before defining the variational problem, we explain the form of its
right-hand side.  Formally, the second-order expression
\[
-D_g(A D_gU)=F
\]
is most naturally interpreted by introducing the flux
\[
P:=A D_gU.
\]
The corresponding first-order system is
\begin{equation}\label{eq:strong-flux-system-motivation}
	A D_gU=P,
	\qquad
	D_gP=-F,
	\qquad
	U(a)=0,
	\qquad
	P(b)=0.
\end{equation}
At this stage \eqref{eq:strong-flux-system-motivation} serves only to motivate
the variational formulation; the strong flux will be reconstructed
rigorously from the variational solution in
\cref{prop:variational-strong}.

Suppose for the moment that \(U\) and \(P\) have the strong regularity
required to apply the lateral integration-by-parts formula.  For
\(\Phi\in V_{g,a}^m([a,b])\), applying \cref{thm:ibp} componentwise to
\(P\) and \(\Phi\) gives
\[
P(b)\cdot\Phi(b)-P(a)\cdot\Phi(a)
=
\int_{[a,b)}P\cdot D_g\Phi\,d\mug
+
\int_{[a,b)}D_gP\cdot\Sg\Phi\,d\mug.
\]
Since
\[
\Phi(a)=0,
\qquad
P(b)=0,
\qquad
D_gP=-F,
\]
we obtain
\begin{equation}\label{eq:variational-motivation}
	\int_{[a,b)}
	(A D_gU)\cdot D_g\Phi\,d\mug
	=
	\int_{[a,b)}
	F\cdot\Sg\Phi\,d\mug.
\end{equation}
Thus the posterior factor in the load term is not inserted by convention:
it is exactly the factor selected by the lateral Stieltjes
integration-by-parts formula when the derivative is transferred from the
flux to the test function.

Motivated by \eqref{eq:variational-motivation}, for
\(F\in L_g^2([a,b);\R^m)\) define
\begin{align}
	a_g(U,\Phi)
	&:=
	\int_{[a,b)}
	(A D_gU)\cdot D_g\Phi\,d\mug,
	\label{eq:ag-form}\\
	\ell_F(\Phi)
	&:=
	\int_{[a,b)}
	F\cdot\Sg\Phi\,d\mug.
	\label{eq:ellF-form}
\end{align}
The preceding estimates show that both expressions are well defined for
\(U,\Phi\in V_{g,a}^m([a,b])\).

\begin{definition}[\(g\)-variational solution]\label{def:g-variational}
	A function \(U\in V_{g,a}^m([a,b])\) is a
	\emph{\(g\)-variational solution} with coefficient \(A\) and datum \(F\) if
	\begin{equation}\label{eq:g-variational}
		a_g(U,\Phi)
		=
		\ell_F(\Phi)
		\qquad
		\text{for every }\Phi\in V_{g,a}^m([a,b]).
	\end{equation}
\end{definition}

\begin{theorem}[Lax--Milgram well-posedness]\label{thm:g-variational}
	Assume \eqref{eq:A-ellipticity}.  For every
	\(F\in L_g^2([a,b);\R^m)\), problem \eqref{eq:g-variational} has a unique
	solution
	\[
	U\in V_{g,a}^m([a,b])
	\subset
	\mathcal W_g^{1,2}([a,b];\R^m).
	\]
	Moreover,
	\begin{equation}\label{eq:variational-stability}
		\|U\|_{V_{g,a}^m([a,b])}
		=
		\|D_gU\|_{L_g^2([a,b);\R^m)}
		\le
		\frac{M_g}{\lambda}
		\|F\|_{L_g^2([a,b);\R^m)},
	\end{equation}
	and
	\begin{equation}\label{eq:variational-L2-stability}
		\|U\|_{L_g^2([a,b);\R^m)}
		\le
		\frac{M_g^2}{\lambda}
		\|F\|_{L_g^2([a,b);\R^m)}.
	\end{equation}
	Consequently, the full strong Sobolev norm satisfies
	\begin{equation}\label{eq:variational-W12-stability}
		\|U\|_{\mathcal W_g^{1,2}([a,b];\R^m)}
		\le
		\frac{M_g(1+M_g)}{\lambda}
		\|F\|_{L_g^2([a,b);\R^m)}.
	\end{equation}
\end{theorem}

\begin{proof}
	By \eqref{eq:A-ellipticity} and Cauchy--Schwarz,
	\[
	|a_g(U,\Phi)|
	\le
	\Lambda
	\|D_gU\|_{L_g^2([a,b);\R^m)}
	\|D_g\Phi\|_{L_g^2([a,b);\R^m)},
	\]
	and therefore
	\[
	|a_g(U,\Phi)|
	\le
	\Lambda
	\|U\|_{V_{g,a}^m([a,b])}
	\|\Phi\|_{V_{g,a}^m([a,b])}.
	\]
	Moreover,
	\[
	a_g(U,U)
	=
	\int_{[a,b)}
	(A D_gU)\cdot D_gU\,d\mug
	\ge
	\lambda
	\|D_gU\|_{L_g^2([a,b);\R^m)}^2
	=
	\lambda
	\|U\|_{V_{g,a}^m([a,b])}^2.
	\]
	Thus \(a_g\) is bounded and coercive on the Hilbert space
	\(V_{g,a}^m([a,b])\).
	
	For the linear functional, Cauchy--Schwarz and
	\eqref{eq:variational-Sg-estimate} give
	\[
	|\ell_F(\Phi)|
	\le
	\|F\|_{L_g^2([a,b);\R^m)}
	\|\Sg\Phi\|_{L_g^2([a,b);\R^m)}
	\le
	M_g
	\|F\|_{L_g^2([a,b);\R^m)}
	\|\Phi\|_{V_{g,a}^m([a,b])}.
	\]
	Hence
	\(\ell_F\in(V_{g,a}^m([a,b]))^*\).  The Lax--Milgram theorem therefore
	yields a unique solution \(U\in V_{g,a}^m([a,b])\); see, for example,
	\cite[Corollary~5.8]{Brezis2011}.
	
	Taking \(\Phi=U\) in \eqref{eq:g-variational} and using coercivity gives
	\[
	\lambda
	\|U\|_{V_{g,a}^m([a,b])}^2
	\le
	|\ell_F(U)|
	\le
	M_g
	\|F\|_{L_g^2([a,b);\R^m)}
	\|U\|_{V_{g,a}^m([a,b])}.
	\]
	If \(U\neq0\), division by
	\(\|U\|_{V_{g,a}^m([a,b])}\) yields
	\eqref{eq:variational-stability}; the estimate is trivial when \(U=0\).
	
	The Poincar\'e--Stieltjes estimate
	\eqref{eq:variational-poincare} then gives
	\[
	\|U\|_{L_g^2([a,b);\R^m)}
	\le
	M_g
	\|U\|_{V_{g,a}^m([a,b])}
	\le
	\frac{M_g^2}{\lambda}
	\|F\|_{L_g^2([a,b);\R^m)},
	\]
	which proves \eqref{eq:variational-L2-stability}.  Finally,
	\eqref{eq:energy-sobolev-norm-equivalence} and
	\eqref{eq:variational-stability} yield
	\[
	\|U\|_{\mathcal W_g^{1,2}([a,b];\R^m)}
	\le
	(1+M_g)
	\|U\|_{V_{g,a}^m([a,b])}
	\le
	\frac{M_g(1+M_g)}{\lambda}
	\|F\|_{L_g^2([a,b);\R^m)},
	\]
	which is \eqref{eq:variational-W12-stability}.
\end{proof}

\subsection{Strong flux reconstruction}

The variational solution admits a completely explicit strong reconstruction.
This is the point at which the first-order character of the preceding theory
becomes useful.

\begin{proposition}[Strong representation of the variational solution]
\label{prop:variational-strong}
Under the hypotheses of \cref{thm:g-variational}, define
\begin{equation}\label{eq:backward-flux}
 P_F(t):=\int_{[t,b)}F(s)\,d\mug(s),
 \qquad t\in[a,b].
\end{equation}
Then \(P_F\in AC_g([a,b];\R^m)\) and
\begin{equation}\label{eq:flux-first-order}
 D_gP_F=-F\quad\mug\text{-a.e.},
 \qquad P_F(b)=0.
\end{equation}
For \(\mug\)-almost every \(t\), the matrix \(A(t)\) is invertible and
\(|A(t)^{-1}\eta|\le\lambda^{-1}|\eta|\).  On the exceptional
\(\mug\)-null set, \(A^{-1}\) is understood to be given an arbitrary measurable
value.  The unique variational solution is therefore
\begin{equation}\label{eq:explicit-strong-U}
 U(t)
 =
 \int_{[a,t)}A(s)^{-1}P_F(s)\,d\mug(s),
 \qquad t\in[a,b].
\end{equation}
Equivalently, with \(P=P_F\),
\begin{equation}\label{eq:g-elliptic-system}
 A D_gU=P,
 \qquad
 D_gP=-F,
 \qquad
 U(a)=0,
 \qquad
 P(b)=0.
\end{equation}
Thus \eqref{eq:g-variational} is the weak form of the mixed Stieltjes flux
problem \(-D_g(A D_gU)=F\), with the second derivative understood through the
first-order flux variable \(P\).
\end{proposition}

\begin{proof}
The backward primitive can be written as
\[
 P_F(t)
 =\int_{[a,b)}F\,d\mug-\int_{[a,t)}F\,d\mug,
\]
so \cref{thm:FTC,def:Dg} yields \eqref{eq:flux-first-order}.  Since
\(\xi^{\mathsf T}A\xi\le |\xi|\,|A\xi|\), the lower bound in
\eqref{eq:A-ellipticity} implies \(|A\xi|\ge\lambda|\xi|\); hence \(A\) is
invertible almost everywhere and \(\|A^{-1}\|\le\lambda^{-1}\).  After
redefining \(A\) on a \(\mug\)-null set if necessary, \(A(t)\in
GL(m,\R)\) for every \(t\); since matrix inversion is continuous on
\(GL(m,\R)\), the map \(A^{-1}\) is measurable.  Furthermore,
\[
 |P_F(t)|
 \le\int_{[a,b)}|F|\,d\mug
 \le M_g^{1/2}\|F\|_{L_g^2([a,b);\R^m)},
\]
so \(A^{-1}P_F\in L_g^2([a,b);\R^m)\).  Formula \eqref{eq:explicit-strong-U} therefore
defines an element of \(V_{g,a}^m([a,b])\) satisfying
\[
 A D_gU=P_F\quad\mug\text{-a.e.}
\]
For every \(\Phi\in V_{g,a}^m([a,b])\), apply \cref{thm:ibp} componentwise to
\(P_F\) and \(\Phi\).  Since \(P_F(b)=0\), \(\Phi(a)=0\), and
\(D_gP_F=-F\),
\[
 \int_{[a,b)}P_F\cdot D_g\Phi\,d\mug
 =
 \int_{[a,b)}F\cdot\Sg\Phi\,d\mug.
\]
Thus the function in \eqref{eq:explicit-strong-U} satisfies the variational
identity.  Uniqueness in \cref{thm:g-variational} identifies it with the
Lax--Milgram solution.
\end{proof}

The construction uses two successive first-order Stieltjes primitives.  This
is the same integral mechanism that underlies the nonhomogeneous first-order
theory in \cite[Section~4.3, Proposition~4.12]{FerMarTo22}.  The difference is
in the boundary orientation: the 2022 result is formulated as an initial-value
problem, whereas here the flux is integrated backward from the natural
condition \(P(b)=0\), and the state is then integrated forward from
\(U(a)=0\).  No claim is made that the precise mixed boundary problem
\eqref{eq:g-elliptic-system} was treated there.

\subsection{An exactly solvable continuous--atomic example}

Proposition~\ref{prop:variational-strong} identifies the Lax--Milgram solution
through two successive first-order Stieltjes primitives.  The following
benchmark makes that abstract reconstruction concrete in the simplest setting
that contains both a continuous part and a genuine atom.  The example has three
purposes: it makes the two atomic jump relations explicit, shows directly why
the variational load is paired with the posterior value \(\Sg\Phi\), and
recovers the classical mixed boundary-value problem when the atom is removed.

\begin{example}[A coercive problem with one atomic jump]
	\label{ex:one-jump-variational}
	
	Let \([a,b]=[0,1]\), fix \(d\in(0,1)\) and \(\alpha>0\), and define
	\begin{equation}\label{eq:one-jump-clock}
		g(t)
		=
		\begin{cases}
			t,
			&0\le t\le d,\\[1mm]
			t+\alpha,
			&d<t\le1.
		\end{cases}
	\end{equation}
	Then \(g\) is nondecreasing and left-continuous, with a single jump at \(d\).
	Its associated Lebesgue--Stieltjes measure is
	\begin{equation}\label{eq:one-jump-measure}
		d\mug
		=
		dt+\alpha\delta_d,
		\qquad
		\Delta g(d)=\alpha.
	\end{equation}
	
	Take the scalar coefficient
	\[
	A=1
	\]
	and the constant forcing
	\[
	F=1.
	\]
	By \cref{thm:g-variational}, there exists a unique
	\[
	U\in V_{g,0}^1([0,1])
	\]
	such that
	\begin{equation}\label{eq:one-jump-variational}
		\int_{[0,1)}
		D_gU\,D_g\Phi\,d\mug
		=
		\int_{[0,1)}
		\Sg\Phi\,d\mug
		\qquad
		\text{for every }
		\Phi\in V_{g,0}^1([0,1]).
	\end{equation}
	
	We first reconstruct this solution through the corresponding strong flux.
	Since \(A=1\) and \(F=1\), the flux is
	\begin{equation}\label{eq:one-jump-flux}
		P(t)
		=
		\int_{[t,1)}1\,d\mug
		=
		g(1)-g(t)
		=
		\begin{cases}
			1+\alpha-t,
			&0\le t\le d,\\[1mm]
			1-t,
			&d<t\le1.
		\end{cases}
	\end{equation}
	In particular,
	\[
	P(1)=0.
	\]
	At the atom \(d\),
	\[
	P(d)=1+\alpha-d,
	\qquad
	P(d^+)=1-d,
	\]
	and therefore
	\begin{equation}\label{eq:one-jump-P-jump}
		P(d^+)-P(d)
		=
		-\alpha
		=
		\Delta g(d)D_gP(d).
	\end{equation}
	Hence
	\[
	D_gP(d)=-1,
	\]
	as required by the strong equation
	\[
	D_gP=-F.
	\]
	
	Integrating the flux once more gives
	\begin{equation}\label{eq:one-jump-solution}
		U(t)
		=
		\int_{[0,t)}P(s)\,d\mug(s)
		=
		\begin{cases}
			\displaystyle
			(1+\alpha)t-\frac{t^2}{2},
			&0\le t\le d,\\[3mm]
			\displaystyle
			t-\frac{t^2}{2}+\alpha+\alpha^2,
			&d<t\le1.
		\end{cases}
	\end{equation}
	Indeed, when \(t>d\), the primitive contains three contributions:
	\[
	\int_{[0,d)}P(s)\,ds,
	\qquad
	\alpha P(d),
	\qquad
	\int_{(d,t)}P(s)\,ds,
	\]
	and the atomic term is
	\[
	\alpha P(d)
	=
	\alpha(1+\alpha-d).
	\]
	This is the origin of the term \(\alpha^2\) in the second branch of
	\eqref{eq:one-jump-solution}.
	
	At the atom,
	\begin{align}
		U(d^+)-U(d)
		&=
		\left(
		d-\frac{d^2}{2}+\alpha+\alpha^2
		\right)
		-
		\left(
		(1+\alpha)d-\frac{d^2}{2}
		\right)
		\notag\\
		&=
		\alpha(1+\alpha-d)
		\notag\\
		&=
		\Delta g(d)D_gU(d).
		\label{eq:one-jump-U-jump}
	\end{align}
	Consequently,
	\[
	D_gU(d)
	=
	P(d)
	=
	1+\alpha-d.
	\]
	Thus the abstract variational solution is exactly the strong Stieltjes
	solution obtained by the two successive primitives
	\[
	D_gP=-1,
	\qquad
	D_gU=P.
	\]
	
	To make the propagation of the atom through the first-order system visible,
	\cref{fig:one-jump-variational} places the clock, the backward flux, and the
	state on the same horizontal scale.  The drawing uses the representative
	values \(d=0.45\) and \(\alpha=0.30\); all labels describing the jumps refer
	to the exact symbolic relations derived above.
	
	\begin{figure}[H]
		\centering
		\begin{tikzpicture}[x=9.6cm,y=1.35cm,>=Latex,font=\small]
			
			% Representative parameters used only for drawing
			\def\d{0.45}
			\def\alp{0.30}
			
			\pgfmathsetmacro{\Pleftd}{1+\alp-\d}
			\pgfmathsetmacro{\Prightd}{1-\d}
			\pgfmathsetmacro{\Ud}{(1+\alp)*\d-0.5*\d*\d}
			\pgfmathsetmacro{\Urightd}{\d-0.5*\d*\d+\alp+\alp*\alp}
			
			% Generous separation between the three aligned panels
			\def\yg{5.40}
			\def\yp{2.70}
			\def\yu{0.00}
			
			% (a) Clock g
			\node[anchor=east,font=\normalsize] at (-0.055,\yg+1.28) {(a)};
			\draw[->] (0,\yg) -- (1.06,\yg) node[right] {$t$};
			\draw[->] (0,\yg) -- (0,\yg+1.48) node[above] {$g(t)$};
			\draw[thick] (0,\yg) -- (\d,\yg+\d);
			\draw[thick] (\d,\yg+\d+\alp) -- (1,\yg+1+\alp);
			\fill (\d,\yg+\d) circle (1.5pt);
			\draw[fill=white] (\d,\yg+\d+\alp) circle (1.5pt);
			\draw[densely dashed] (\d,\yg) -- (\d,\yg+\d+\alp);
			\draw[decorate,decoration={brace,amplitude=3.5pt,mirror}]
			(\d+0.035,\yg+\d) -- (\d+0.035,\yg+\d+\alp)
			node[midway,right=4pt] {$\Delta g(d)=\alpha$};
			\node[below] at (\d,\yg) {$d$};
			\node[above] at (0.84,\yg+0.84+\alp) {$g$};
			
			% (b) Backward flux P
			\node[anchor=east,font=\normalsize] at (-0.055,\yp+1.28) {(b)};
			\draw[->] (0,\yp) -- (1.06,\yp) node[right] {$t$};
			\draw[->] (0,\yp) -- (0,\yp+1.48) node[above] {$P(t)$};
			\draw[thick] (0,\yp+1+\alp) -- (\d,\yp+\Pleftd);
			\draw[thick] (\d,\yp+\Prightd) -- (1,\yp);
			\fill (\d,\yp+\Pleftd) circle (1.5pt);
			\draw[fill=white] (\d,\yp+\Prightd) circle (1.5pt);
			\draw[densely dashed] (\d,\yp) -- (\d,\yp+\Pleftd);
			\draw[decorate,decoration={brace,amplitude=3.5pt,mirror}]
			  (\d+0.035,\yp+\Prightd) -- (\d+0.035,\yp+\Pleftd)
			  node[midway,right=4pt] {$\alpha$};
			\node[below] at (\d,\yp) {$d$};
			\node[above left=1pt] at (\d,\yp+\Pleftd) {$P(d)$};
			\node[below right=1pt] at (\d,\yp+\Prightd) {$P(d^+)$};
			\node[above] at (0.78,\yp+0.22) {$P$};
			
			% (c) State U
			\node[anchor=east,font=\normalsize] at (-0.055,\yu+1.18) {(c)};
			\draw[->] (0,\yu) -- (1.06,\yu) node[right] {$t$};
			\draw[->] (0,\yu) -- (0,\yu+1.38) node[above] {$U(t)$};
			\draw[thick,domain=0:\d,samples=80,smooth]
			  plot (\x,{\yu+(1+\alp)*\x-0.5*\x*\x});
			\draw[thick,domain=\d:1,samples=80,smooth]
			  plot (\x,{\yu+\x-0.5*\x*\x+\alp+\alp*\alp});
			\fill (\d,\yu+\Ud) circle (1.5pt);
			\draw[fill=white] (\d,\yu+\Urightd) circle (1.5pt);
			\draw[densely dashed] (\d,\yu) -- (\d,\yu+\Urightd);
			\draw[decorate,decoration={brace,amplitude=3.5pt,mirror}]
			  (\d+0.035,\yu+\Ud) -- (\d+0.035,\yu+\Urightd)
			  node[midway,right=4pt] {$\alpha P(d)$};
			\node[below] at (\d,\yu) {$d$};
			\node[below left=1pt] at (\d,\yu+\Ud) {$U(d)$};
			\node[above right=1pt] at (\d,\yu+\Urightd) {$U(d^+)$};
			\node[above] at (0.82,\yu+0.84) {$U$};
			
		\end{tikzpicture}
		
		\caption{One-jump benchmark for the representative values
			\(d=0.45\) and \(\alpha=0.30\).  Filled points denote values at
			\(d\), and open points the posterior values at \(d^+\).  From top to
			bottom, the signed jumps of the clock, flux, and state are
			\(+\alpha\), \(-\alpha\), and \(+\alpha P(d)\), respectively.  The TikZ code used to render this figure was prepared with assistance from
			OpenAI ChatGPT and verified by the author.}
		\label{fig:one-jump-variational}
	\end{figure}
	
	The aligned dashed lines show how the same atom propagates through the two
	first-order equations.  The jump \(\Delta g(d)=\alpha\) produces
	\(P(d^+)-P(d)=-\alpha\) because \(D_gP(d)=-1\), and the flux value at the
	atom then produces \(U(d^+)-U(d)=\alpha P(d)\) because
	\(D_gU(d)=P(d)\).  The figure is only a visualization of these relations;
	the next calculation verifies algebraically the posterior factor required in
	the variational identity.
	
	The posterior value appearing in
	\eqref{eq:one-jump-variational} can now be checked directly at the atom.
	Since
	\[
	D_gP(d)=-1,
	\]
	the atomic increment of the product \(P\Phi\) is
	\begin{align*}
		P(d^+)\Phi(d^+)-P(d)\Phi(d)
		&=
		P(d)\bigl(\Phi(d^+)-\Phi(d)\bigr)
		+
		\bigl(P(d^+)-P(d)\bigr)\Phi(d^+)
		\\[1mm]
		&=
		\alpha P(d)D_g\Phi(d)
		-
		\alpha\Sg\Phi(d).
	\end{align*}
	This is precisely the atomic contribution of the lateral
	integration-by-parts identity applied to \((P,\Phi)\).
	
	Equivalently, using
	\[
	\Sg\Phi(d)
	=
	\Phi(d^+)
	=
	\Phi(d)+\alpha D_g\Phi(d),
	\]
	the atomic contribution can be written as
	\[
	\alpha P(d)D_g\Phi(d)
	-
	\alpha\Phi(d)
	-
	\alpha^2D_g\Phi(d).
	\]
	If one incorrectly replaced the posterior value
	\(\Sg\Phi(d)=\Phi(d^+)\) by the anterior value \(\Phi(d)\), the term
	\[
	-\alpha^2D_g\Phi(d)
	\]
	would be lost.  Thus the posterior factor in the forcing term is already
	forced by the local algebra of this one-atom example.
	
	For completeness, the variational identity can also be verified directly
	from the strong flux.  Since
	\[
	D_gU=P,
	\qquad
	D_gP=-1,
	\qquad
	P(1)=0,
	\qquad
	\Phi(0)=0,
	\]
	the lateral integration-by-parts formula gives
	\[
	0
	=
	P(1)\Phi(1)-P(0)\Phi(0)
	=
	\int_{[0,1)}
	P\,D_g\Phi\,d\mug
	-
	\int_{[0,1)}
	\Sg\Phi\,d\mug.
	\]
	Therefore
	\[
	\int_{[0,1)}
	D_gU\,D_g\Phi\,d\mug
	=
	\int_{[0,1)}
	\Sg\Phi\,d\mug,
	\]
	which is exactly \eqref{eq:one-jump-variational}.
	
	Finally, if \(\alpha=0\), then the atomic part disappears,
	\[
	g(t)=t,
	\qquad
	\mug=dt,
	\qquad
	\Sg=I,
	\]
	and \eqref{eq:one-jump-solution} reduces to
	\[
	U(t)
	=
	t-\frac{t^2}{2}.
	\]
	This is exactly the classical solution of
	\[
	-U''=1,
	\qquad
	U(0)=0,
	\qquad
	U'(1)=0.
	\]
	Hence the example simultaneously verifies the classical limit, the atomic
	jump relations, the appearance of the posterior factor in the variational
	formulation, and the equivalence between the Lax--Milgram and strong
	Stieltjes constructions.
	
\end{example}

\section{Concluding remarks}\label{sec:conclusion}

The first-order Stieltjes--Sobolev theory admits a global weak formulation in
which the classical weak--strong correspondence is preserved once the atomic
orientation is made explicit.  The fundamental theorem supplies the strong
integral representative, the Stieltjes integration-by-parts identity supplies
the weak pairing, and the \(L_g^1([a,b))\)--\(L_g^\infty([a,b))\) separation argument
recovers the representative.  The only modification in the test pairing is
the posterior value \(\Sg\Phi\) at atoms; for continuous clocks, including
singular-continuous ones, \(S_g\) reduces to the identity.

Endpoint information also has a precise structural role.  The initial value
fixes the integration constant.  The final value is determined by the strong
fundamental theorem, but at the level of a bare \(L_g^p([a,b))\)-class it can retain
one posterior increment that becomes invisible after a last positive atom
with no subsequent measure mass.  The terminal-defect proposition isolates
exactly this case and shows when the reduced state is sufficient.

The derivative-measure identity
\[
 DU=(D_gU)\mug
\]
places the formulation directly in the classical \(BV\) framework.  The
coercive mixed problem in \cref{sec:variational} gives a further functional
consequence: its \(g\)-variational form is well posed by Lax--Milgram and the
same solution is recovered explicitly from the first-order flux pair
\(A D_gU=P\), \(D_gP=-F\), with \(U(a)=0\) and \(P(b)=0\).  The one-jump
example makes this equivalence completely explicit and shows, within a direct
application of the variational theorem, how the posterior factor fixes the
atomic forcing orientation.

The framework developed here is therefore a compact reference formulation
for weak first-order Stieltjes differentiation relative to one derivator.
Extensions to genuinely higher-order weak derivatives, singular right-hand
sides not absolutely continuous with respect to \(\mug\), Banach-valued
evolution, or variation of the clock require additional structures and are
left for separate work.

\section*{Acknowledgements}
The author is grateful to Fernando Adri\'an Fern\'andez Tojo for his valuable
comments and suggestions, which have substantially improved the manuscript.

\section*{Statements and Declarations}

\paragraph{Funding.}
The author was supported by the Xunta de Galicia through the project
``Consolidaci\'on e Estruturaci\'on 2023 GRC GI-1561---Ecuaci\'ons diferenciais non lineais (EDNL).''

\paragraph{Competing interests.}
The author has no relevant financial or non-financial interests to disclose.

\paragraph{Data availability.}
No datasets were generated or analysed during the current study.

\section*{Author information}
Francisco Javier Fern\'andez Fern\'andez, Department of Statistics, Mathematical
Analysis and Optimisation, Faculty of Mathematics, Universidade de Santiago de
Compostela, R\'ua Lope G\'omez de Marzoa, s/n, Campus Vida, 15782 Santiago de
Compostela, A Coru\~na, Spain.  E-mail: \texttt{fjavier.fernandez@usc.es}.


\begin{thebibliography}{99}

\bibitem{Agarwal2006}
R.~P. Agarwal, V. Otero--Espinar, K. Perera and D.~R. Vivero,
Basic properties of Sobolev's spaces on time scales,
\emph{Adv. Difference Equ.} \textbf{2006} (2006), Article ID 38121, 14 pp.,
doi:10.1155/ADE/2006/38121.

\bibitem{Bartle1995}
R.~G. Bartle,
\emph{The Elements of Integration and Lebesgue Measure},
Wiley, New York, 1995.

\bibitem{Brezis2011}
H. Brezis,
\emph{Functional Analysis, Sobolev Spaces and Partial Differential Equations},
Springer, New York, 2011,
doi:10.1007/978-0-387-70914-7.

\bibitem{EckhardtTeschl2012}
J. Eckhardt and G. Teschl,
On the connection between the Hilger and Radon--Nikodym derivatives,
\emph{J. Math. Anal. Appl.} \textbf{385} (2012), 1184--1189,
doi:10.1016/j.jmaa.2011.07.041.


\bibitem{FerMarTo22}
F.~J. Fern\'andez, I. M\'arquez Alb\'es and F.~A.~F. Tojo,
On first and second order linear Stieltjes differential equations,
\emph{J. Math. Anal. Appl.} \textbf{511} (2022), Article 126010,
doi:10.1016/j.jmaa.2022.126010.

\bibitem{FerMarTo24}
F.~J. Fern\'andez, I.~M\'arquez Alb\'es and F.~A.~F. Tojo,
On second-order linear Stieltjes differential equations with non-constant
coefficients,
\emph{Open Math.} \textbf{22} (2024), 20240018,
doi:10.1515/math-2024-0018.

\bibitem{FerToVil24}
F.~J. Fern\'andez, F.~A.~F. Tojo and C. Villanueva,
Compactness criteria for Stieltjes function spaces and applications,
\emph{Results Math.} \textbf{79} (2024), Article 98,
doi:10.1007/s00025-024-02132-4.

\bibitem{FriLo17}
M. Frigon and R. L\'opez Pouso,
Theory and applications of first-order systems of Stieltjes differential equations,
\emph{Adv. Nonlinear Anal.} \textbf{6} (2017), 13--36.



\bibitem{LoMa18}
R. L\'opez Pouso and I. M\'arquez Alb\'es,
General existence principles for Stieltjes differential equations with applications to mathematical biology,
\emph{J. Differential Equations} \textbf{264} (2018), 5388--5407.

\bibitem{LoRo14}
R. L\'opez Pouso and A. Rodr\'iguez,
A new unification of continuous, discrete, and impulsive calculus through Stieltjes derivatives,
\emph{Real Anal. Exchange} \textbf{40} (2014/15), 319--354.

\bibitem{MarraffaSatcoBVP2025}
V. Marraffa and B. Satco,
Stieltjes differential problems with general boundary value conditions:
existence and bounds of solutions,
\emph{Czechoslovak Math. J.} \textbf{75} (2025), 235--255,
doi:10.21136/CMJ.2024.0125-23.

\bibitem{MarraffaSatcoSecond2025}
V. Marraffa and B. Satco,
The method of lower and upper solutions for second order periodic Stieltjes differential equations,
\emph{J. Fixed Point Theory Appl.} \textbf{27} (2025), Article 9,
doi:10.1007/s11784-025-01162-x.

\bibitem{MaMon21}
I. M\'arquez Alb\'es and G.~A. Monteiro,
Notes on the existence and uniqueness of solutions of Stieltjes differential equations,
\emph{Math. Nachr.} \textbf{294} (2021), 794--814.

\bibitem{MarSlaTvr23}
I. M\'arquez Alb\'es, A. Slav\'ik and M. Tvrd\'y,
Duality for Stieltjes differential and integral equations,
\emph{J. Math. Anal. Appl.} \textbf{519} (2023), Article 126789,
doi:10.1016/j.jmaa.2022.126789.

\bibitem{MonteiroSatco2017}
G.~A. Monteiro and B. Satco,
Distributional, differential and integral problems: equivalence and existence results,
\emph{Electron. J. Qual. Theory Differ. Equ.} \textbf{2017} (2017), No.~7, 1--26,
doi:10.14232/ejqtde.2017.1.7.

\bibitem{MonSlaTvr18}
G.~A. Monteiro, A. Slav\'ik and M. Tvrd\'y,
\emph{Kurzweil--Stieltjes Integral: Theory and Applications},
World Scientific, Singapore, 2018,
doi:10.1142/9432.

\bibitem{SimasSousa2025}
A.~B. Simas and K.~J.~R. Sousa,
One-sided measure theoretic elliptic operators and applications to SDEs driven by Gaussian white noise with atomic intensity,
\emph{Potential Anal.} \textbf{63} (2025), 1347--1380,
doi:10.1007/s11118-025-10208-1.

\end{thebibliography}
\end{document}